\documentclass[letterpaper, 10 pt, conference]{ieeeconf}  

\IEEEoverridecommandlockouts                              
\usepackage{mathrsfs}
\usepackage{placeins}

\usepackage{amsmath,amsthm,amssymb,mathtools,bbm,dsfont,physics}

\usepackage[hidelinks]{hyperref}
\usepackage[capitalize]{cleveref}

\usepackage{xcolor}

\usepackage{comment}

\usepackage{etoolbox}
\newbool{booltrue} 
\booltrue{booltrue}
\newbool{boolfalse} 
\boolfalse{boolfalse}
\usepackage{hyperref}
\usepackage{graphicx}
\usepackage{subcaption}
\newtheorem{theorem}{Theorem}
\newtheorem{lemma}[theorem]{Lemma}
\newtheorem{corollary}[theorem]{Corollary}

\newtheorem{definition}{Definition}

\newtheorem{remark}{Remark}

\newcommand{\tk}[1]{#1}

\newcommand{\Htwo}{\mathit{H_2}}
\newcommand{\Hinf}{\mathit{H_\infty}}

\newcommand{\DRL}{\text{DR-LQR}}
\newcommand{\FixAlg}{\texttt{Fixed-Point} }
\newcommand{\BL}{\Gamma}

\renewcommand{\u}{\mathsf{u}} 
\newcommand{\w}{\mathsf{w}} 
\newcommand{\x}{\mathsf{x}} 

\newcommand{\F}{\mathcal{F}} 
\newcommand{\G}{\mathcal{G}} 
\newcommand{\K}{\mathcal{K}} 
\newcommand{\T}{\mathcal{T}} 
\newcommand{\NN}{\mathcal{N}}
\newcommand{\M}{\mathcal{M}} 
\renewcommand{\L}{\mathcal{L}} 
\newcommand{\I}{\mathcal{I}} 
\newcommand{\U}{\mathcal{U}} 

\newcommand{\ufin}{\mathbf{u}} 
\newcommand{\wfin}{\mathbf{w}} 
\newcommand{\xfin}{\mathbf{x}} 

\newcommand{\Ffin}{\mathbf{F}} 
\newcommand{\Gfin}{\mathbf{G}} 
\newcommand{\Kfin}{\mathbf{K}} 
\newcommand{\Tfin}{\mathbf{T}} 
\newcommand{\CCfin}{\mathbf{C}} 
\newcommand{\Ifin}{\mathbf{I}} 

\newcommand{\W}{\mathscr{W}} 
 
\newcommand{\causal}{\mathscr{K}}

\newcommand{\BW}{\mathsf{BW}} 
\newcommand{\Was}{\mathsf{W_2}}

\newcommand{\cost}{\operatorname{cost}}

\newcommand{\ejw}{\e^{{j}\omega}}
\newcommand{\emjw}{\e^{\-{j}\omega}}

\newtheorem{problem}{Problem}

\newcommand{\defeq}{\coloneqq}                      
\newcommand{\ghi}{\gamma_{\infty}}

\newcommand{\+}{\!+\!}
\renewcommand{\-}{\!-\!}
\renewcommand{\=}{\!=\!}

\renewcommand{\>}{\!>\!}

\renewcommand{\leqq}{\!\leq\!}

\newcommand{\ie}{\textit{i.e.}}

\newcommand{\Z}{\mathbb{Z}}     
\newcommand{\R}{\mathbb{R}}     
\newcommand{\C}{\mathbb{C}}     

\newcommand{\Sym}{\mathbb{S}}   

\newcommand{\Prob}{\mathscr{P}} 

\newcommand{\pr}[1]{\left({#1}\right)}          
\newcommand{\br}[1]{\left[{#1}\right]}          
\newcommand{\cl}[1]{\left\{{#1}\right\}}        

\newcommand{\Abs}[1]{\left\vert{#1}\right\vert}         

\newcommand{\Norm}[2][\text{}]{\left\Vert{#2}\right\Vert_{#1}} 

\newcommand{\inner}[2]{\langle{#1},\,{#2}\rangle}

\newcommand{\clf}[1]{\mathcal{#1}} 

\newcommand{\tp}{\intercal}                 

\newcommand{\inv}{\mathrm{\-1}}             

\newcommand{\psdleq}{\preccurlyeq} 
\newcommand{\psdg}{\succ}          

\newcommand{\kron}{\otimes} 

\newcommand{\E}{\operatorname{\mathbb{E}}} 
\renewcommand{\Pr}{\operatorname{\mathbb{P}}} 

\newcommand{\sampled}[1][\text{}]{\stackrel{#1}{\sim}}

\newcommand{\beq}[1]{\begin{align*}\label{eq:#1}}
\newcommand{\eeq}{\end{align*}}

\newcommand{\suml}{\sum\nolimits}

\newcommand{\e}{\mathrm{e}}             
\newcommand{\half}{\frac{1}{2}} 

\makeatletter
\newcommand{\xMapsto}[2][]{\ext@arrow 0599{\Mapstofill@}{#1}{#2}}
\def\Mapstofill@{\arrowfill@{\Mapstochar\Relbar}\Relbar\Rightarrow}
\makeatother

\RequirePackage{algorithm}
\RequirePackage{algorithmic}

\title{\LARGE \bf
The Distributionally Robust Infinite-Horizon LQR}

\author{Joudi Hajar, Taylan Kargin, Vikrant Malik, Babak Hassibi
\thanks{}
\thanks{}
\thanks{The authors are with the Electrical Engineering Department at Caltech. Their emails are {\tt\small \{jhajar,tkargin,vmalik,hassibi\}@caltech.edu}}%
}

\begin{document}

\maketitle
\thispagestyle{empty}
\pagestyle{empty}

\begin{abstract}
We explore the infinite-horizon Distributionally Robust (DR) linear-quadratic control. While the probability distribution of disturbances is unknown and potentially correlated over time, it is confined within a Wasserstein-2 ball of a radius $r$ around a known nominal distribution. Our goal is to devise a control policy that minimizes the worst-case expected Linear-Quadratic Regulator (LQR) cost among all probability distributions of disturbances lying in the Wasserstein ambiguity set. We obtain the optimality conditions for the optimal DR controller and show that it is non-rational. Despite lacking a finite-order state-space representation, we introduce a computationally tractable fixed-point iteration algorithm. Our proposed method computes the optimal controller in the frequency domain to any desired fidelity. Moreover, for any given finite order, we use a convex numerical method to compute the best rational approximation (in $H_\infty$-norm) to the optimal non-rational DR controller. This enables efficient time-domain implementation by finite-order state-space controllers and addresses the computational hurdles associated with the finite-horizon approaches to DR-LQR problems, which typically necessitate solving a Semi-Definite Program (SDP) with a dimension scaling with the time horizon. We provide numerical simulations to showcase the effectiveness of our approach.

\end{abstract}

\section{Introduction}
How to deal with uncertainty is a core challenge in decision-making. Control systems inherently encounter various uncertainties, such as external disturbances, measurement errors, model disparities, and temporal variations in dynamics \cite{grinten1968uncertainty,doyle1985structured}. Neglecting these uncertainties in policy design can result in considerable performance decline and may lead to unsafe and unintended behavior \cite{samuelson2017}.

Traditionally, the challenge of uncertainty in control systems has been predominantly approached through either the stochastic or robust control frameworks \cite{kalman_new_1960,zames_feedback_1981,doyle_state-space_1988}. Stochastic control, as seen in Linear–Quadratic–Gaussian (LQG) or $H_2$ control, aims to minimize an expected cost, assuming disturbances follow a known probability distribution \cite{blackbook}. However, in many practical scenarios, the true distribution is often estimated from sampled data, introducing vulnerability to inaccurate models. On the other hand, robust control minimizes the worst-case cost across potential disturbance realizations, such as those with bounded energy or power ($H_\infty$ control) \cite{zhou_robust_1996}. While this ensures robustness, it can be overly conservative. 

To tackle the above challenge, a recent approach called {\em Distributionally Robust (DR) Control} has emerged. In contrast to traditional approaches such as $H_2$ or $H_\infty$ control that focus on a single probability distribution or on a worst-case disturbance realization, the DR framework addresses the uncertainty in system dynamics and disturbances by considering ambiguity sets – sets of plausible probability distributions \cite{yang2020wasserstein,kim_distributional_2021, hakobyan2022wasserstein, tacskesen2023distributionally, aolaritei_wasserstein_2023,aolaritei_capture_2023,brouillon2023distributionally}. This methodology aims to design controllers that exhibit robust performance across all probability distributions within a given ambiguity set. The size of the ambiguity set provides control over the desired robustness against distributional uncertainty, ensuring that the resulting controller is not excessively conservative.

Different measures of distributional mismatch, such as total variation \cite{tzortzis_dynamic_2014,tzortzis_robust_2016} and KL divergence \cite{liu_data-driven_2023}, are explored in DR control. However, for computational feasibility, ambiguity sets are commonly defined as Wasserstein-2 balls around a nominal distribution \cite{mohajerin_esfahani_data-driven_2018,gao_distributionally_2022}. This choice ensures a practical way, since the optimization of quadratic costs over Wasserstein-2 balls leads to semi-definite programs, to connect the domains of stochastic and adversarial uncertainties.

\subsection{Contributions}
In this study, we explore the Wasserstein-2 distributionally robust LQR (DR-LQR) control framework. DR-LQR control seeks to design controllers that minimize the worst-case expected cost across distributions chosen adversarially within a Wasserstein-2 ambiguity set. Our contributions are summarized as follows.

\paragraph{Stabilizing Infinite-Horizon Controller.} Rather than the finite-horizon setting prevalent in the DR control literature \cite{hakobyan2022wasserstein, tacskesen2023distributionally,aolaritei_capture_2023,DRORO,hajar_wasserstein_2023}, we focus on the infinite-horizon setting. 
Thus, we provide long-term stability and robustness guarantees.
\vspace{0mm}

\paragraph{Robustness to Arbitrarily Correlated Disturbances.} Unlike several prior works which assume time-independence of the disturbances \cite{yang2020wasserstein,kim_distributional_2021, hakobyan2022wasserstein, tacskesen2023distributionally, zhong2023nonlinear,aolaritei_wasserstein_2023,aolaritei_capture_2023}, we do not impose such assumptions so that the resulting controllers are robust against time-correlated disturbances. 
\vspace{0mm}

\paragraph{Computationally Efficient Controller Synthesis}
Leveraging a strong duality result, we obtain the exact Karush-Khun-Tucker (KKT) conditions for the worst-case distribution and the optimal causal controller.
We show that, although the resulting controller is non-rational, \ie, it does not admit a finite state-space form, it can still be computed very efficiently.
We provide a computationally efficient numerical method to compute the optimal non-rational DR-LQR controller in the frequency domain via fixed-point iterations. We further show how to find the best rational approximation, and thereby the best finite-dimensional state-space controller, for any given degree. Prior works focus on finite horizon problems (see \cite{DRORO,hajar_wasserstein_2023, tacskesen2023distributionally}) and therefore have no stability guarantees. More importantly, they are hampered by the fact that they require solving a semi-definite program (SDP) whose size scales with the time horizon. This prohibits their applicability when the time horizon is large. Our approach enables efficient implementation of the infinite-horizon DR-LQR controller. 

\paragraph{Comparison to \cite{kargin2023wasserstein}} Earlier work of the authors studied the problem of infinite horizon distributionally robust regret-optimal (DR-RO) control \cite{kargin2023wasserstein}. The DR-RO control is similar to the LQR problem considered here, since the cost in both cases is quadratic. In the context of distributionally-robust control in \cite{kargin2023wasserstein}, the regret-optimal control problem (originally studied in \cite{hajar_regret-optimal_2023,sabag2023regretoptimal}) is much simpler since the cost of the optimal non-causal controller is removed from the LQR cost. In this paper we show that, despite the more complicated form of the quadratic cost, the main results of \cite{kargin2023wasserstein} extend to the LQR case. The KKT conditions, the fixed-point iterative algorithm, and the final controller are different, but the general methodology goes through.

\section{{Preliminaries and Problem Setup} } \label{sec:prelim}

\subsection{Notations}
Going forward, calligraphic letters ($\K$, $\M$, $\L$, etc.) represent infinite-horizon operators, while boldface letters ($\Kfin$, $\CCfin$, $\mathbf{w}$, etc.) denote finite-horizon operators. $\I$ and $\Ifin$ are the identity operators. $\M^\ast$ is the adjoint of $\M$, and $\psdg$ denotes the positive-definite ordering. The trace functions for finite and infinite horizon operators are denoted by $\tr(\cdot)$ and $\Tr(\cdot)$, respectively, where $\Tr(\I) = p$ for a finite horizon. The Euclidean norm is denoted by $\norm{\cdot}$, while $\norm{\cdot}_\infty$ and $\norm{\cdot}_2$ refer to the $\Hinf$ (operator) and $\Htwo$ (Frobenius) norms, respectively. ${\M}{+}$ and ${\M}{-}$ are the causal and strictly anti-causal parts of $\M$. The notation $\sqrt{\M}$ or $\M^{\half}$ indicates the symmetric positive square root. Matrices in the set of positive semi-definite matrices are represented by $\Sym_+^n$, and $[\cdot]_T$ signifies the finite-horizon restriction of operators. $\abs{z}$ is the magnitude and $z^\ast$ is the conjugate of a complex number $z\!\in\!\C$. The complex unit circle is denoted by $\mathbb{T}$. Finally, $\sigma$ denotes the singular value, $\text{id}$ denotes the identity map, and $\times$ the cartesian product.

\subsection{Linear-Quadratic Control}
Consider the state-space representation of discrete-time linear time-invariant (LTI) dynamical system:
\vspace{-3mm}

\begin{equation}\label{eq:state_space}
\begin{aligned}
    x_{t+1} &= A x_{t} + B_u u_{t} + B_w w_{t},
\end{aligned}
\end{equation}
Here, $x_{t} \in \R^{n}$ denotes the state, $u_t \in \R^{d}$ is the control input, 
and $w_{t} \in \R^{p}$ is the disturbance. We posit that both the pairs $(A,B_u)$ and $(A,B_w)$ are stabilizable in the usual sense. For a finite horizon $T\>0$, the system in~\eqref{eq:state_space} incurs a quadratic cost as
\begin{equation}\label{eq:cost}
    \cost_T(\ufin, {\wfin}) \defeq \suml_{t=0}^{T-1} x_t^\tp Q x_t + u_t^\tp R u_t,
\end{equation}
where $Q,R \psdg 0$. Without loss of generality, we let $Q=I$, $R=I$ by redefining $x_{t} \leftarrow Q^{\half} x_{t}$ and $u_{t} \leftarrow R^{\half} u_{t}$. 
\vspace{0mm}

\paragraph{System in Operator Form}We opt for operator notation for system \eqref{eq:state_space} in the rest of this paper. For a given horizon $T > 0$, we let the sequences $\xfin\! \defeq\! \{x_{t}\}_{t=0}^{T-1}$, $\ufin \!\defeq\! \{u_{t}\}_{t=0}^{T-1}$, and $\wfin\! \defeq \!\{w_{t}\}_{t=0}^{T-1}$, represent the state, control input and exogenous disturbances, respectively. Likewise, we represent their infinite-horizon counterparts using the bi-infinite sequences $\x \defeq \{x_{t}\}_{t\in \Z }$, $\u \defeq \{u_{t}\}_{t\in \Z }$, and $\w \defeq \{w_{t}\}_{t\in \Z }$.

Using the above definitions, we express the system dynamics in both finite and infinite horizon in operator form as
\begin{equation}
\begin{aligned}\label{eq:operator_form}
\text{Finite-horizon: }\,\xfin = \Ffin \ufin + \Gfin \wfin,\\
\text{Infinite-horizon: }\,\x = \F \u + \G \w. 
\end{aligned}
\end{equation}
where $(\F,\G)$ denote strictly causal (strictly lower triangular) bi-infinite block Toeplitz operators and $(\Ffin,\Gfin)$ represent their finite horizon equivalents, for a horizon $T > 0$. Employing this notation, we succinctly express the LQR cost in \cref{eq:cost} as $ \cost_T(\ufin, {\wfin}) \defeq \norm{\xfin}^2 + \norm{\ufin}^2$.

\vspace{0mm}

\paragraph{Control} We focus on linear disturbance feedback control policies (DFC) which map disturbances to the control input: $\ufin = \Kfin \wfin$, for any $\Kfin \in \causal_T$, where $\causal_T$ is the set of \emph{causal} (online) DFC policies in the finite-horizon of length $T\>0$. The infinite-horizon counterpart of our control policy is $\u = \K \w$ for any $\K \in \causal$, with $\causal$ the set of \emph{causal and time-invariant} DFC policies in the infinite horizon.

Under a fixed control policy $\K$, the closed-loop transfer operator, $\T_{\K}$, which maps the disturbances to the state and control input, is defined as
\vspace{-0mm}
\begin{equation} \label{eq:T_K}
\vspace{-0mm}
    \T_{\K} : \w \mapsto \begin{bmatrix} \x \\ \u \end{bmatrix} \defeq \begin{bmatrix} \F \K + \G \\ \K \end{bmatrix} \w.
\end{equation}
The finite-horizon counterpart of the  closed-loop transfer operator~\eqref{eq:T_K} denoted as $\Tfin_{\Kfin}$, is used to rewrite the quadratic cost~\eqref{eq:cost} as \begin{equation}
    \cost_T(\Kfin, {\wfin}) = {\wfin^\ast {\Tfin_{\Kfin}^\ast \Tfin_{\Kfin}} \wfin}.
\end{equation}

\subsection{Distributionally Robust LQR Control}
We study the distributionally robust LQR control problem, and seek to design a causal controller which minimizes the worst-case expected LQR cost when the probability distributions of the disturbances reside  within a Wasserstein-2 ($\Was$) ambiguity set. The \textit{Wasserstein-2 metric} between two distributions $\Pr_1,\Pr_2$ is defined as 
\begin{equation*}\label{eq:wasserstein}
    \Was(\Pr_1,\Pr_2) \!\defeq\!\! \pr{\inf_{\pi \in \Pi(\Pr_1,\Pr_2)} \int \norm{w_1 \- w_2}^2 \,\pi(dw_1,dw_2)}^{\half} ,
\end{equation*}
with $\Pi(\Pr_1,\Pr_2)$ representing all the joint distributions with marginals $\Pr_1$ and $\Pr_2$ \cite{villani_optimal_2009, wassOT2}. 

We define the $\Was$-ambiguity set $\W_T(\Pr_\circ,r)$ for horizon $T\>0$ to be the $\Was$-ball with radius  $r_T \defeq r\sqrt{T}>0$, centered at a nominal distribution $\Pr_\circ \in \Prob(\R^{p T})$:
\begin{equation}\label{eq:wass ambiguity set} 
    \W_T(\Pr_\circ,r_T) \!\defeq\! \cl{\Pr\! \in\! \Prob(\R^{p T}) \!\mid \! \Was(\Pr,\, \Pr_\circ) \leqq r_T}.
\end{equation}

Unlike the usual LQR cost, which considers the expected quadratic cost under i.i.d gaussian disturbances (or disturbances sampled from a single probability distribution), the DR-LQR considers the worst-case expected LQR cost across all disturbance probability distributions which lie in the $\Was$-ambiguity set.

\begin{definition}[\textbf{Worst-case expected LQR cost under $\Was$-ambiguity}] \label{def:worst case cost}

The worst-case expected LQR cost of a control policy $\Kfin \in \causal_T$, in the finite-horizon $T\>0$, is defined as 
\begin{equation}\label{eq:worst case exp cost finite}
    C_T(\Kfin,r_T) \defeq\!\! \sup_{\Pr \in \W_T(\Pr_\circ,r_T)} \E_{\Pr}\br{\cost_T(\Kfin, {\wfin})}.
\end{equation}
Likewise, in the infinite-horizon, the worst-case expected LQR cost of a control policy $\K \in \causal$  is defined as
\begin{equation}\label{eq:worst case exp cost infinite}
    C(\K,r) \defeq \lim_{T\to \infty}  \frac{1}{T}  \, C_T([\K]_T,r_T),   
\end{equation}
where $[\K]_T$ results from restricting $\K$ to a horizon $T\>0$. In these definitions, $\E_{\Pr}$ is the expectation over the disturbances ${\wfin} $ which are sampled from $\Pr$: ${\wfin} \!\sampled\! \Pr$. 
\end{definition}

We formally state the infinite-horizon Distributionally Robust LQR problem as follows:
\begin{problem}[\textbf{Distributionally Robust LQR Control in the Infinite-Horizon}]\label{prob:DR-RO}
    Minimize the time-averaged worst-case expected cost~\eqref{eq:worst case exp cost infinite} as $T\to\infty$, over all causal and time-invariant controllers $\K\in\causal$, \ie,
    \begin{equation} \label{eq:DR-RO}
    \inf_{\K \in \causal} {C}(\K,r) = \inf_{\K \in \causal}  \lim_{T\to \infty}  \frac{1}{T}  \, C_T([\K]_T,r_T) .
\end{equation}
\end{problem}
In section \ref{sec:main_results}, we provide an equivalent formulation of \cref{prob:DR-RO} by establishing strong duality for the worst-case expected cost $C(\K,r)$.

\section{Main Theoretical Results}\label{sec:main_results}
In this section, our initial step involves reformulating Problem \ref{prob:DR-RO} through the lens of strong duality, breaking it down into a task of addressing a suboptimal problem. We proceed to characterize the controller using the Karush-Kuhn-Tucker (KKT) conditions and present arguments to show that it is stabilizing. 



For simplicity of the presentation, we assume $\M_\circ = I$.

\begin{theorem}[Strong Duality] \label{thm:full_optimality}
The distributionally robust LQR control problem \eqref{eq:DR-RO} is equivalent to the dual optimization problem: 
\begin{equation}
    \inf_{\substack{\K \in\causal \\ \gamma \geq 0}} \sup_{\M \psdg 0} \gamma ( r^2 +\Tr  (2\sqrt{\M}-\M - \I)  +  \gamma^{-1}\Tr(\T_{\K}^\ast \T_{\K} \M) )
    \label{eq:full_optimality}
\end{equation}

Additionally, with a fixed $\K$, the worst case disturbance $\w_\star$ can be given using the nominal disturbance $\w_\circ$ as $\w_\star = (\I - \gamma_\star^{-1} \T_{\K}^\ast \T_{\K})^{-1} \w_{\circ}$, with $\gamma_\star$ satisfying:
\begin{equation}\label{eq:worst_gamma}
\Tr((\I - \gamma_\star^{-1} \T_{\K}^\ast \T_{\K} )^{-1} - \I)^2 = r^2.
\end{equation}
\end{theorem}
\begin{proof}
This theorem leverages the proof of Theorem 5 and Lemma 8 from \cite{kargin2023wasserstein}, adapted here to incorporate our LQR cost, $\T_{\K}^\ast \T_{\K}$.
\end{proof}
This insight reveals that the essence of distributional robustness can be distilled into the analysis of the worst-case power spectrum of the disturbances, $\M$, which deviates by no more than $r > 0$ from the baseline disturbance spectrum $\M_\circ$.

Moreover, for any chosen $r > 0$, a corresponding optimal $\gamma > 0$ exists that facilitates the computation of the optimal controller. Given the feasibility of searching across a singular parameter $\gamma > 0$, our focus shifts towards the $\gamma$-optimal problem once $\gamma$ is fixed.

Let $\Delta^\ast \Delta = \I + \F^\ast \F$ be the spectral factorization with causal $\Delta$, and $\Delta^\inv$. And, let $\K_\circ \defeq - (\I + \F^\ast \F)^{\-1}\F^\ast \G$ be the unique optimal non-causal policy which minimizes the infinite-horizon cost $\lim_{T\to \infty} \frac{1}{T} \cost_{T}(\Kfin,{\wfin})$\cite{ sabag2021regret}, with $\T_{\K_\circ}$ its associated closed-loop transfer operator \eqref{eq:T_K}.  

With the above, we can now give the saddle point conditions for the controller $\K$ and the disturbance covariance $\M$ in Theorem \ref{thm:kkt}, for a fixed $\gamma\geq 0$:

\begin{theorem}[{{\textbf{$\gamma$-optimal solution via saddle points } }}]\label{thm:kkt}
Given $\gamma >{\gamma}_{H_\infty}\defeq  \inf_{\K\in\causal} \norm{\T_{\K}^\ast \T_{\K}}$ be fixed. The $\gamma$-optimal LQR control problem in \cref{thm:full_optimality} for fixed $\gamma$ is equivalent to the following dual problem:
\begin{equation}\label{eq:suboptimal_prob}
      \sup_{\M \psdg 0} \inf_{\K \in\causal} \Tr(2\sqrt{\M}-\M) + \gamma ^\inv  \Tr(\T_{\K}^\ast \T_{\K} \M ).
\end{equation}
And for $\M_\circ = \I$, the unique saddle point $(\K_{\gamma}, \M_{\gamma})$ satisfies the following equations uniquely:
\begin{align}
    \K_{\gamma} &\= \Delta^{\inv} \{\Delta \K_\circ \L_{\gamma} \}_{\!+} \L_{\gamma}^{\inv} \label{eq:K from L in KKT}\\
    \L_{\gamma}^\ast \L_{\gamma}  &\=  \frac{1}{4}\left(\I \+ \sqrt{\I \+4 \gamma^\inv (\mathcal S_{\L_\gamma}^\ast \mathcal S_{\L_\gamma} + \mathcal U_{\L_{\gamma}}^\ast \mathcal U_{\L_{\gamma}}) }\right)^2\label{eq:worst_case_N}
\end{align}
where $\L_{\gamma} \L_{\gamma}^\ast = \M_{\gamma} $ with causal $\L_{\gamma}$, and $\L_{\gamma}^\inv$, and where $\mathcal S_{\L_\gamma}$ and $\mathcal U_{\L_\gamma}$ are defined as $\mathcal S_{\L_\gamma} \defeq \{\Delta \K_\circ \L_\gamma\}_{\-}$ , $\mathcal U_{\L_\gamma}\defeq \mathcal T_{\K_\circ}\L_\gamma $.
\end{theorem}
\begin{proof}
    The proof is built upon the KKT conditions for \eqref{eq:suboptimal_prob} and the Wiener-Hopf technique \cite{kailath_linear_2000}. Details of the proof are shown in Appendix \ref{ap:pfopt}.
\end{proof}

Theorem \ref{thm:kkt} hints at a way to obtain the optimal controller $\K_{\gamma}$ from a positive operator $\M_{\gamma}$ \eqref{eq:K from L in KKT}. Since the optimality conditions on $\M_{\gamma}$ can be solely expressed in terms of its spectral factors and system-specific parameters as in \eqref{eq:worst_case_N}, we will shift our focus to obtaining a solution $\M_{\gamma}$, which is the \emph{the worst-case time-invariant covariance operator.}
\begin{remark}
    Denoting $\NN_{\gamma} \defeq  \L_{\gamma}^\ast \L_{\gamma}$, we note that there is a one-to-one correspondence between $\M_{\gamma}$ and $\NN_{\gamma}$ through spectral factorization. As the optimality conditions in \cref{thm:kkt} are stated in terms of $\L_{\gamma}^\ast \L_{\gamma}$, we call both $\NN_{\gamma}$ and $\M_{\gamma}$ as the $\gamma$-optimal solution, interchangeably. 
\end{remark}

\begin{remark}\label{remark:limiting_r}
When $r$ approaches infinity,  $\gamma_\star$ approaches the lower bound $\norm{\T_{\K}^\ast \T_{\K}}_{op}$, which corresponds to the worst-case cost, or in other words, the $H_\infty$ cost. In this scenario, the optimal DR-LQR controller transitions into the traditional $H_\infty$ controller. On the flip side, as $r$ decreases to zero, $\gamma_\star$ goes to infinity. This shift results in the worst-case expected LQR cost, $C(\K_{\gamma_\star}, r)$, aligning with the anticipated cost under nominal disturbance, $\M_\circ$, thereby allowing the optimal DR-LQR controller to embody the conventional LQR ($\Htwo$) controller, particularly when $\M_\circ = \I$. The process of adjusting $r$ thus enables the DR-LQR controller to navigate a spectrum between $H_\infty$ and $H_2$ control paradigms.
\end{remark}

This insight reveals that once $\gamma$ surpasses ${\gamma}_{H\infty}$, the expected worst-case LQR cost becomes finite. This observation leads to the formulation of Corollary \ref{thm:stabilizable}.

\begin{corollary}\label{thm:stabilizable}
For any chosen $\gamma$ value exceeding ${\gamma}_{H_\infty}$, the resulting suboptimal controller, $\K_\gamma$, ensures the stabilization of the system's dynamics.
\end{corollary}
  


\section{Algorithm for Irrational Controller Synthesis in the Frequency Domain}

In this section, we assert that the sub-optimal DR-LQR controller is \emph{irrational}, thereby precluding a finite-dimensional state-space realization. Consequently, we introduce a fixed-point iteration scheme aimed at computing the saddle-point solution $(\K_{\gamma}, \NN_{\gamma})$ for the $\gamma$-optimal problem outlined in \cref{thm:kkt}, with fixed $\gamma$. The optimal $\gamma_\star$ and its associated saddle-point solution $(\K_{\gamma_\star}, \NN_{\gamma_\star})$, can be effectively determined by employing the bisection method on equation~\eqref{eq:worst_gamma}. Finally, we prove the convergence of the fixed-point method in the scalar system case, \ie~$p=d=n=1$.

\subsection{Fixed-Point Iteration to solve for the Controller}\label{subsec::subK}

Let the frequency domain counterparts of the operators $\mathcal S, \mathcal U_L, \text{and }\NN$ be $S_{L_\gamma}(z) \defeq \{\Delta K_\circ L_\gamma\}_{\-}(z)$, $U_{L_\gamma}(z) \defeq T_{K_\circ}(z)L_\gamma(z) $ and $N_\gamma(z) \defeq {L_\gamma(z)}^\ast L_\gamma(z)$. We use the fact that $\{\clf{Y}\}_{\+}\=\clf{Y}\-\{\clf{Y}\}_{\-}$ to express the KKT equations~\eqref{eq:worst_case_N} in the frequency domain as:
\begin{align}
    K_\gamma(z) &= K_{\circ}(z)\- \Delta^\inv (z) S_{L_\gamma}(z) L_\gamma^\inv (z) \label{eq:K_freq},\\
    N_\gamma(z) &= \!\frac{1}{4}\!\pr{I \!\+ \!\sqrt{\!I \+ 4 \gamma^\inv\!(S^\ast_{L_\gamma}\!(z) S_{L_\gamma}\!(z)+U_{L_\gamma}^\ast\!(z) U_{L_\gamma}\!(z)) }}^{\!2} \label{eq:N}
\end{align}

\vspace{-1mm}The anticausal transfer function $S_{L_\gamma}(z)$ has the following state-space form representation (as shown in \cite{kargin2023wasserstein}): $S_{L_\gamma}(z) \defeq \overline{C}(z^{\-1}I - \overline{A})^\inv \overline{B}_{L_\gamma}$, with  $\overline{B}_{L_\gamma}\defeq \frac{1}{2\pi} \int_{0}^{2\pi} (I-\ejw \overline{A})^\inv \overline{D} L_\gamma(\ejw) d\omega$, where $(\overline{A}, \overline{C}, \overline{D})$ depend on the system parameters $(A,B,C)$ (see the appendix \ref{app1:def1} for the full definitions).

Given the aforementioned notation, we now introduce the following theorem, characterizing the $\gamma$-optimal solution as a fixed point.

\begin{theorem}[$\gamma-$optimal solution is a fixed-Point Solution]\label{thm:fixed-point}
For a fixed $\gamma \> \gamma_{H_\infty}$,  consider the following set of mappings: \tk{}
\begin{align}\label{eq:fixedpointmaps}
    &F_1: L(z) \mapsto  \overline{B}_L\defeq \frac{1}{2\pi} \int_{0}^{2\pi} (I-z \overline{A})^\inv \overline{D} L(\ejw) d\omega\\   
    &F_{2,\gamma}:(\overline{B}_L,L(z)) \mapsto N(z), \nonumber\\
    &N(z) \defeq \frac{1}{4}\pr{I \+ \sqrt{I \+ 4 \gamma^\inv (S_L^\ast(z) S_L(z)+U_L(z)^\ast U_L(z)) }}\nonumber\\
     &\text{with } S_L(z) = \overline{C}(z^{\-1}I - \overline{A})^\inv \overline{B}_L, \text{ }U_L(z)=T_{K_\circ}(z)L(z) \\ 
    &F_3: N(z) \mapsto  L(z), \quad 
\end{align}
where $F_3$ returns a unique spectral factor of $N(z)>0$. The composition $F_3 \!\circ\! F_{2,\gamma} \!\circ\! (F_1 \times\text{id} ) : L(z) \mapsto  L(z)$ admits a unique fixed-point $L(z)$, and  $N_{\gamma}(z)~\defeq~F_{2,\gamma} \!\circ\! (F_1 \times\text{id} ) (L_\gamma(z))$  satisfies the KKT conditions~\eqref{eq:N}.
\end{theorem}
\begin{proof}
    The proof is similar to the proof of theorem 13 in \cite{kargin2023wasserstein}. It utilizes the concavity of the problem in $\mathcal{M_\gamma}$ to argue for uniqueness of $\mathcal{M_\gamma}$, and thus of its spectral factor $\mathcal{L_\gamma}$ up to a unitary transformation, which leads to a unique fixed-point.
\end{proof}

Subsequently, we argue that $N_\gamma(z)$ is irrational. Note that $S_{L_\gamma}(z)$ is rational, and assuming $L_\gamma(z)$ is rational implies that $U_{L_\gamma}(z)$ is rational. Hence, $N_\gamma(z)$ involves the square-root of a rational term. As square root does not preserve rationality in general, both $N_\gamma(z)$ and its spectral factor $L_\gamma(z)$ are irrational, leading to a contradiction.  This yields Corollary~\ref{thm:irrational}.
\begin{corollary}\label{thm:irrational}
     $N_\gamma(z)$ and the suboptimal \DRL~controller, $K_\gamma(z)$, are irrational, for any fixed $\gamma\in(\gamma_{H_\infty},\infty)$. Hence, $K_\gamma(z)$ does not have a finite-dimensional state-space representation.
\end{corollary}

Despite the fact that $K_\gamma(z)$ does not lend itself to a finite-dimensional state-space form, Theorem~\ref{thm:fixed-point} affirms that the suboptimal controller $K_\gamma(z)$ \eqref{eq:K_freq} can be derived from $L_\gamma(z)$, by executing a fixed-point iteration on $L_\gamma(z)$ as elucidated in Section \ref{subsec:alg}.

\subsection{Algorithm Description}\label{subsec:alg}

In light of Theorem~\ref{thm:fixed-point}, we propose Algorithm \FixAlg to compute the suboptimal controller $K_\gamma(z)$ at uniformly sampled points on the unit circle, $\mathbb{T}_N \!\defeq\! \{\e^{j 2\pi n /N} \mid n\=0,\!\dots\!,N\-1\}$. With an initial estimate $L_{\gamma}^{(0)}(z)$, \FixAlg iteratively computes the $n$-th step as $N^{(n)}_\gamma(z)=F_{2,\gamma} \!\circ\! (F_1 \times\text{id} )(L_\gamma^{(n)}(z))$.

Following this, we compute the spectral factor $L_\gamma^{(n+1)}(z)$  at regularly spaced points along the unit circle using the \texttt{SpectralFactor} algorithm. Upon reaching convergence within a predetermined tolerance at the $N$-th iteration, we determine the suboptimal $N_\gamma^{(N)}(z)$ from which we derive the \emph{suboptimal} controller $K_\gamma^{(N)}(z)$ at each sampled frequency point using~\eqref{eq:K_freq}. 

Note that this approach can assess $N_\gamma^{(n)}(z)$ for any arbitrary point on the unit circle via \cref{eq:N}, however, the \texttt{SpectralFactor} algorithm computes $L_\gamma^{(n)}(z)$ and $K_\gamma^{(n)}(z)$ only for a discrete number of samples on the unit circle. The details of the \texttt{SpectralFactor} algorithm, which is based on discrete Fourier transform (DFT) \cite{rino_factorization_1970} and tailored for \emph{irrational} spectra, are in Appendix \ref{ap:spectral factorization}.

\begin{algorithm}[tb]
   \caption{\FixAlg  }
   \label{alg:fixed_point}
\begin{algorithmic}
   \STATE {\bfseries Input:} $\gamma\!\>\!{\gamma}_{\,H_\infty}$, system $(\overline{A}, \overline{C}, \overline{D})$, discretization $N$
   
   Initialize $L_{\gamma}^{(0)}(z), \forall z \!\in\! \mathbb{T}_N = \{\e^{j 2\pi n /N} \mid n=0,\!...\!,N-1\}$
   \REPEAT 
   \STATE Compute $\overline{B}_{L_\gamma}^{(n)}=F_1(L_\gamma^{(n)}(z))$ numerically $$\overline{B}_{L_\gamma}^{(n)} \gets \frac{1}{N} \sum_{z\in \mathbb{T}_N}  (I-z \overline{A})^{-1}\overline{D} L_\gamma^{(n)}(z) $$
   \vspace{-3mm}
   \STATE Compute $N_\gamma^{(n)}(z) \leftarrow  F_{2,\gamma}(\overline{B}_{L_\gamma}^{(n)},L_\gamma^{(n)}(z))$
   \STATE Get $L_\gamma^{(n+1)}(z) \leftarrow \texttt{SpectralFactor}( N_\gamma^{(n)}(z))$
   \STATE Update $n\leftarrow n+1$
   \UNTIL{convergence of $N_\gamma^{(n)}(z)$}
\end{algorithmic}
\end{algorithm}

\subsection{Convergence of \FixAlg }\label{sec:cvg}

Though extensive simulations show exponential convergence of \FixAlg, we provide a convergence proof in the case of a scalar system. As we showed in theorem \ref{thm:fixed-point} that there exists a unique fixed point to our algorithm, the proof sketch for convergence follows the following arguments. Note that this proof is for the special case when $p = d=n = 1$.

\begin{itemize}
    \item We consider two spectrums $\NN_1$ and $\NN_2$.
    \item Then, we consider the change that one iteration does to these spectrums. That is, after one iterations, the spectrums are $\bar{\NN_1}$ and $\bar{\NN_2}$.
    \item We show that if $\NN_1 \preccurlyeq \NN_2$, then $\bar{\NN_1} \preccurlyeq \bar{\NN_2}$. That is, the iterative algorithm preserves the order of the spectrums.
    \item We can now initialize the algorithm with $\NN = 0$. Since, our iterates satisfy: $\NN \succcurlyeq 0$, the spectrums are monotonically increasing at each iteration.
    \item Since there is a unique fixed point, the algorithm must converge to the optimal $\NN$.
\end{itemize}
The complete proof details are shown in Appendix \ref{ap:pfcvg}.

\section{Approximating the non-rational Controller by a Rational Controller with a State-Space Structure}\label{sec:rational_approx}
\subsection{Rational Approximation}

Our objective is to find the optimal $m^{th}$ order rational approximation, denoted as $P(z)/Q(z)$, for the positive non-rational function $N(z)$. This rational approximation serves as the basis for obtaining the spectral factor $L(z)$ as noted in lemma \ref{lemma:RA} and thus deriving the controller $K(z)$ \eqref{eq:K_freq}. We provide results for scalar disturbances, \ie~$p=1$, while the states and control inputs can be arbitrarily dimensional vectors. We leave the generalization to vector disturbances (\ie, $p>1$) for future work.

\begin{problem}
[Rational approximation using $H_\infty$ norm minimization] 
\label{pb:RA}Given a positive non-rational function $N(z)$ , find the best rational approximation of order at most  $m \in \mathbb{N}$ with respect to $H_\infty$-norm:
\begin{align}\label{eq:obj_RA}
    &\inf_{\substack{p_0,\dots,p_m, q_0,\dots,q_m \in \mathbb{R}}} \Norm[\infty]{\frac{P(z)}{Q(z)} - N(z) }
\end{align}
with $P(z)=\sum_{k=-m}^{m} p_k z^{-k}$, $p_k = p_{-k} \in \R$, and $P(z)> 0$, (and similarly for $Q(z)$).
\end{problem}

To solve problem \ref{pb:RA} using standard convex optimization tools, we follow the approach in \cite{kargin2024infinitehorizondistributionallyrobustregretoptimal} and  consider instead the sublevel sets of the objective function \eqref{eq:obj_RA} and reduce the problem to a convex feasibility problem. 

\begin{lemma}[Rational approximation using a convex feasibility problem]\label{lem:hinf_affine}Fixing a minimum level $\epsilon>0$, problem \ref{pb:RA} can be relaxed to a convex problem:
\begin{align*}
    &\begin{aligned}
        &\text{Find } \mathbf{p} = (p_0,p_1,\dots,p_m),\, \mathbf{q} = (q_0,q_1,\dots,q_m) \\
        &\text{s.t } P(z),Q(z)\geq 0, \, \max_{z\in \mathbb{T}} \left| \frac{P(z)}{Q(z)} - N(z) \right| \leq \epsilon, \\
        &\textrm{or equivalently}\\
        &\text{s.t } \begin{cases}
            P(z) - (N(z) + \epsilon) Q(z) \leq 0, \quad \forall z\in\mathbb{T} \\
            P(z) - (N(z) - \epsilon) Q(z) \geq 0, \quad \forall z\in\mathbb{T}\\
            P(z),Q(z)\geq 0, \quad \forall z\in\mathbb{T}
        \end{cases} 
    \end{aligned} \label{eq:hinf_affine}
    \end{align*}
\end{lemma}

Although the inequalities in Lemma \ref{lem:hinf_affine} are infinitely many, we can check these inequalities solely for a finite set of frequencies, such as $\mathbb{T}_N = \{\e^{j 2\pi n /N} \mid n\=0,\!\dots\!,N\-1\}$ for $N\gg m$.  In fact,  the finite polynomials $P(z)$ and $Q(z)$ can be fully characterized with $N\geq 2m$ number of uniformly sampled frequencies on the unit circle by Nyquist sampling theorem. By increasing the number of samples, the accuracy of this method can be improved to any desired fidelity. 

Once obtained a rational approximation $P(z) / Q(z)$ for $N(z)$, we can find the rational canonical factor $L(z)$ of $P(z) / Q(z)$ from the following lemma.

\begin{lemma}[Canonical Factorization \cite{sayed_survey_2001}]\label{lemma:RA}
Given a polynomial of order $m$, $R(z)=\sum_{k=-m}^{m} r_k z^{-k}$, where $r_k = r_{-k} \in \R$, and $R(z)> 0$, a causal canonical factor, $L(z)= \sum_{k=0}^{m} l_k z^{-k}$, exists and satisfies $R(z) = |L(z)|^2$.
\end{lemma}

\subsection{Controller in Time-Domain}

Once we have the rational approximation of $L(z)$ as $L(z)=(I+\Tilde{C} (z I -\Tilde{A})^{-1}\Tilde{B})\Tilde{D}^{1/2}$, we can compute the \DRL~controller in state-space form.
\begin{lemma}[\DRL~control in state-space form]\label{lemma:rationalK}
    Given the rational approximation of $L(z)$ as $L(z)=(I+\Tilde{C} (z I -\Tilde{A})^{-1}\Tilde{B})\Tilde{D}^{1/2}$, the \DRL~controller is given by
    \begin{align}
        e(t+1)&=\Tilde{F} e(t)+\Tilde{G} w(t)\\ 
        u(t)&=\Tilde{H} e(t)+ \Tilde{J} w(t) \label{eq:RationalK2}
    \end{align}
where $(\Tilde{F},\Tilde{G},\Tilde{H},\Tilde{J})$ are functions of the matrices $(A,B,C,\Tilde A,\Tilde B, \Tilde C )$ (see Appendix \ref{app:controller} for the details of the appropriate definitions, and Appendix \ref{ap:pfstate} for the complete proof).
\end{lemma}
\section{Numerical Simulations}\label{sec::numerical}
This section presents a comparative evaluation of the DR-LQR controller vis-à-vis $H_2$ and $H_\infty$ controllers, alongside the finite-horizon DR-LQR counterpart. Our evaluation encompasses both frequency domain and time-domain assessments which showcase the efficacy of the rational approximation method. Our analysis focuses on benchmark models from \cite{aircraft} such as [AC15], [REA4] and [HE3]. Given the similarity in controller performance across all systems, we opt to present results solely for [AC15], a four-state aircraft model, due to space limitations. We choose our nominal distribution to be Gaussian, with zero mean and identity covariance.

\subsection{Frequency Domain Evaluations}
We examine the dynamics of the DR-LQR controller and its rational approximation across varying radii $r$.

The power spectrum $N(e^{j\omega})$ of the worst-case disturbance is illustrated for three distinct $r$ values for the [AC15] system in Figure \ref{fig:fig_m_freq_1}. Notably, for $r=0.01$, the worst-case disturbance exhibits near-white behavior, consistent with the nominal disturbance. However, as $r$ increases, the temporal correlation of the worst-case disturbance intensifies, leading to a more pronounced peak in the power spectrum.

\begin{figure}[htbp]
	\centering		 \includegraphics[width=0.42\textwidth]{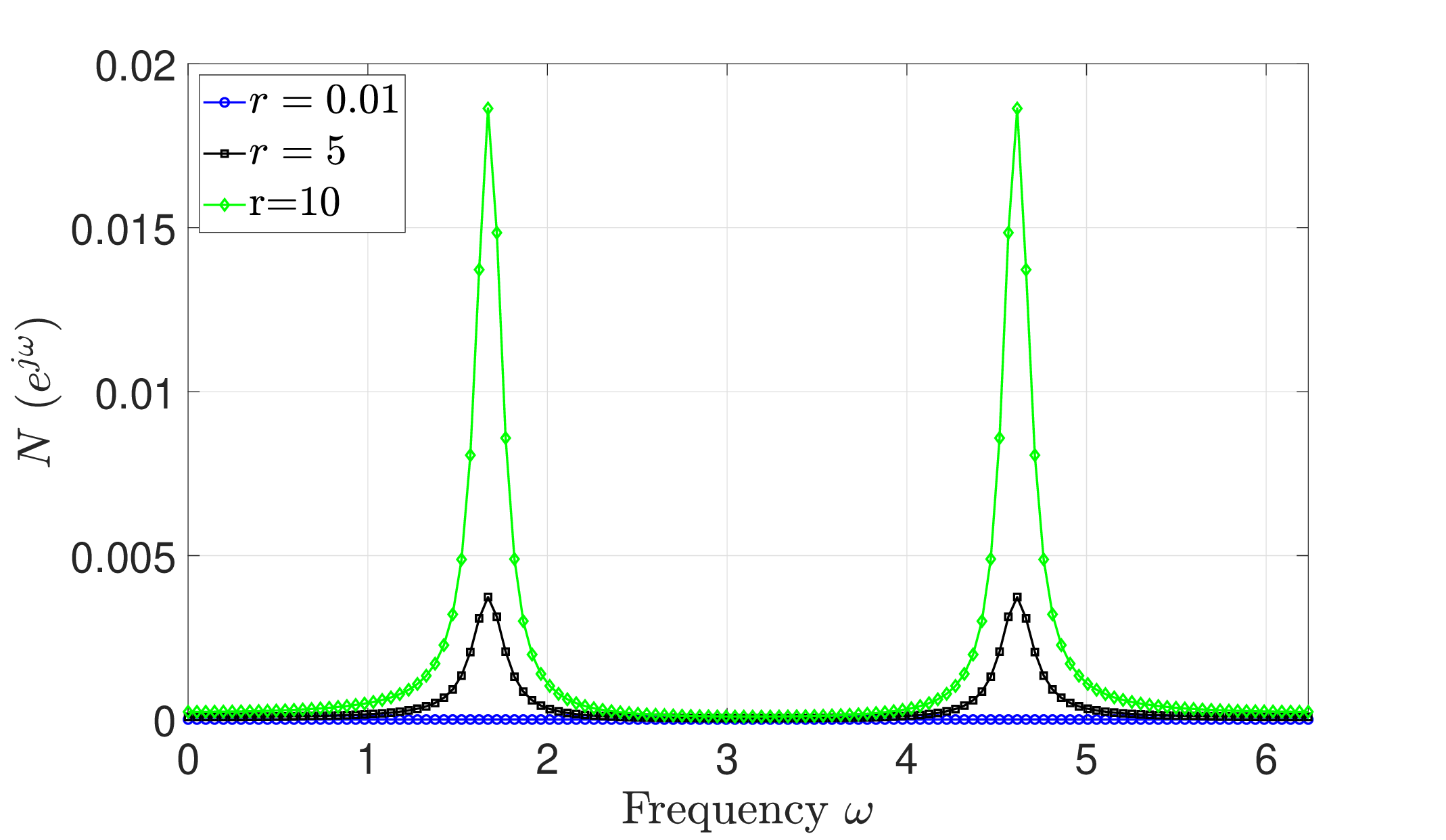}
       \caption{The power spectrum $N(\ejw)$ of the worst-case disturbance, when $r \in (0.01, 5, 10)$ for system [AC15]. 
        }
		\label{fig:fig_m_freq_1}
\end{figure}

Figures \ref{fig:ER} and \ref{fig:PercDiff} illustrate the worst-case expected LQR cost for the DR-LQR, $H_2$, and $H_\infty$ controllers applied to the [AC15] system. As $r$ varies, the performance of the DR-LQR closely mirrors that of the ${H}_2$ for smaller values of $r$. However, with increasing $r$, the worst-case LQR cost tends to align more closely with that of the $H_\infty$ controller. Across all ranges of $r$, the DR-LQR consistently outperforms the other controllers and achieves the lowest worst-case expected cost.
\begin{figure}[htbp]
		\centering		\includegraphics[width=0.42\textwidth]{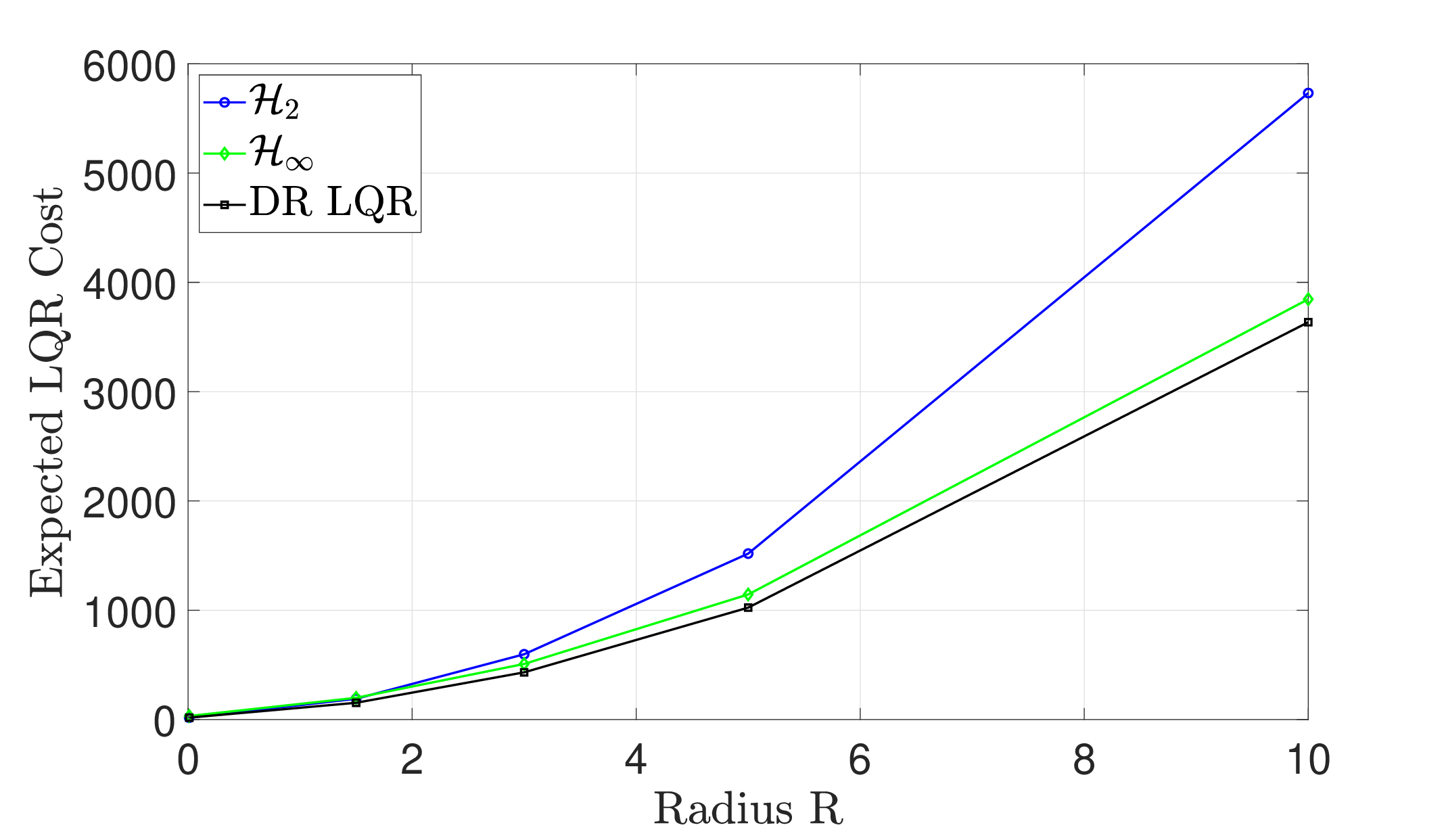} \caption{The worst-case expected LQR cost of the classical controllers $H_2$, $H_\infty$ compared to the DR-LQR, for the system [AC15], for different $r$ values. The DR-LQR minimizes the cost at all $r's$. 
      }
    \label{fig:ER}
\end{figure}

\begin{figure}[htbp]
        \vspace{-5mm}
		\centering		\includegraphics[width=0.42\textwidth]{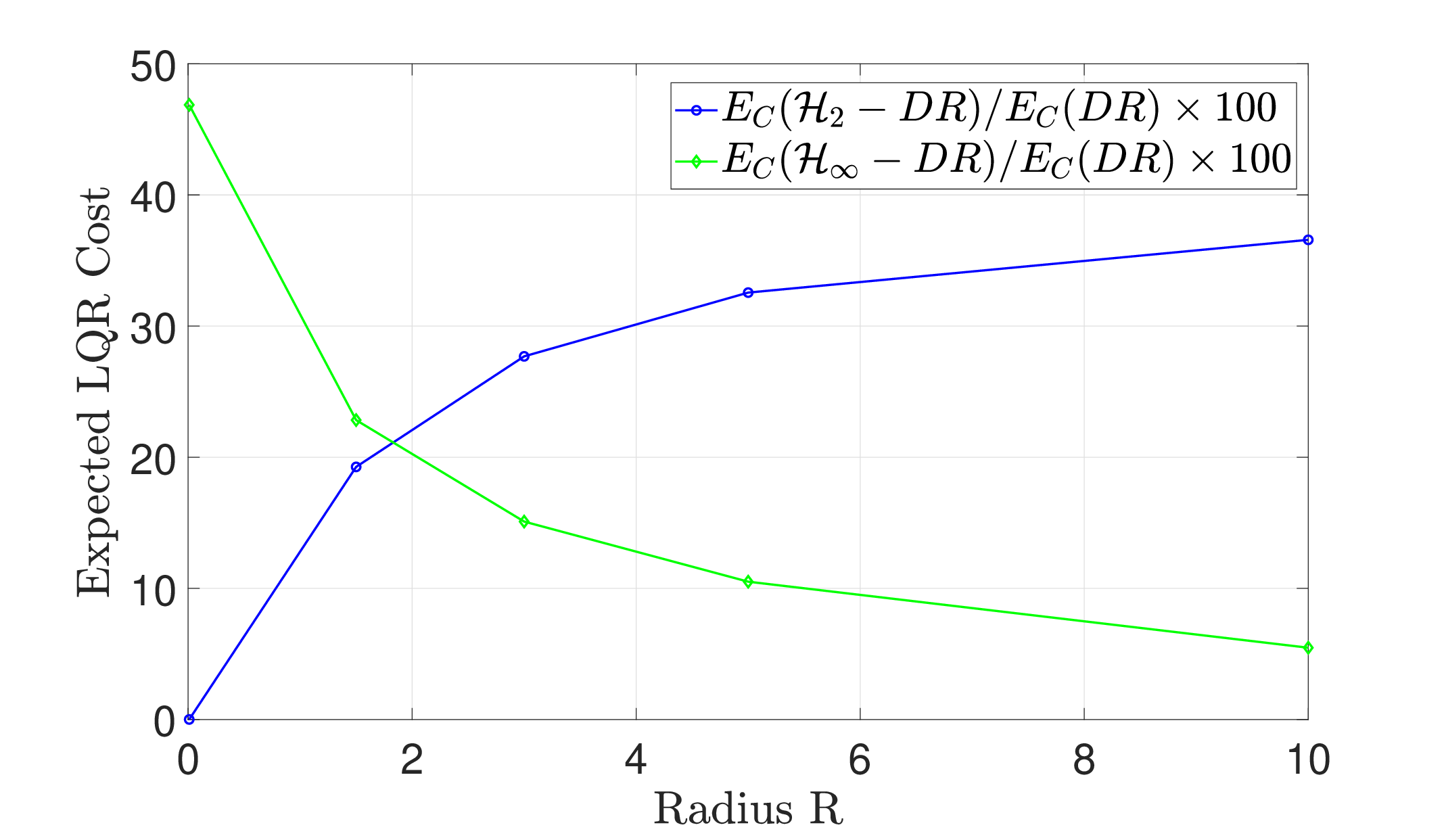} \caption{The percentage difference in the worst case LQR cost relative to the DR-LQR (see the legend) for the system [AC15], for different values of $r$. When $r$ is small (large) $r$, the cost of DR-LQR controller closely aligns with that of ${H}_2$ (${ H}_\infty$). For $r=1.5$, the cost of the DR-LQR is less than that of $H_2$ by $22.8\%$, and that of $H_\infty$ by $19.2\%$ .	
      }
    \label{fig:PercDiff}
\end{figure}
Additionally, we consider another performance metric—the operator norm of $\T_\K$ minimized by the $H_\infty$ controller. This metric, expressed in the frequency domain as $\|\T_\K \|_{op}^2 = \max_{0 \leq \omega \leq 2\pi} \sigma_{\max}( T_K^\ast(\ejw)T_K(\ejw))$, is depicted across all frequencies in Figure \ref{fig:R}. The results show that the DR controller demonstrates a consistent interpolation between the $H_2$ and $H_\infty$ controllers across all frequencies.

\begin{figure}[htbp]
    \centering
    \includegraphics[width=0.42\textwidth]{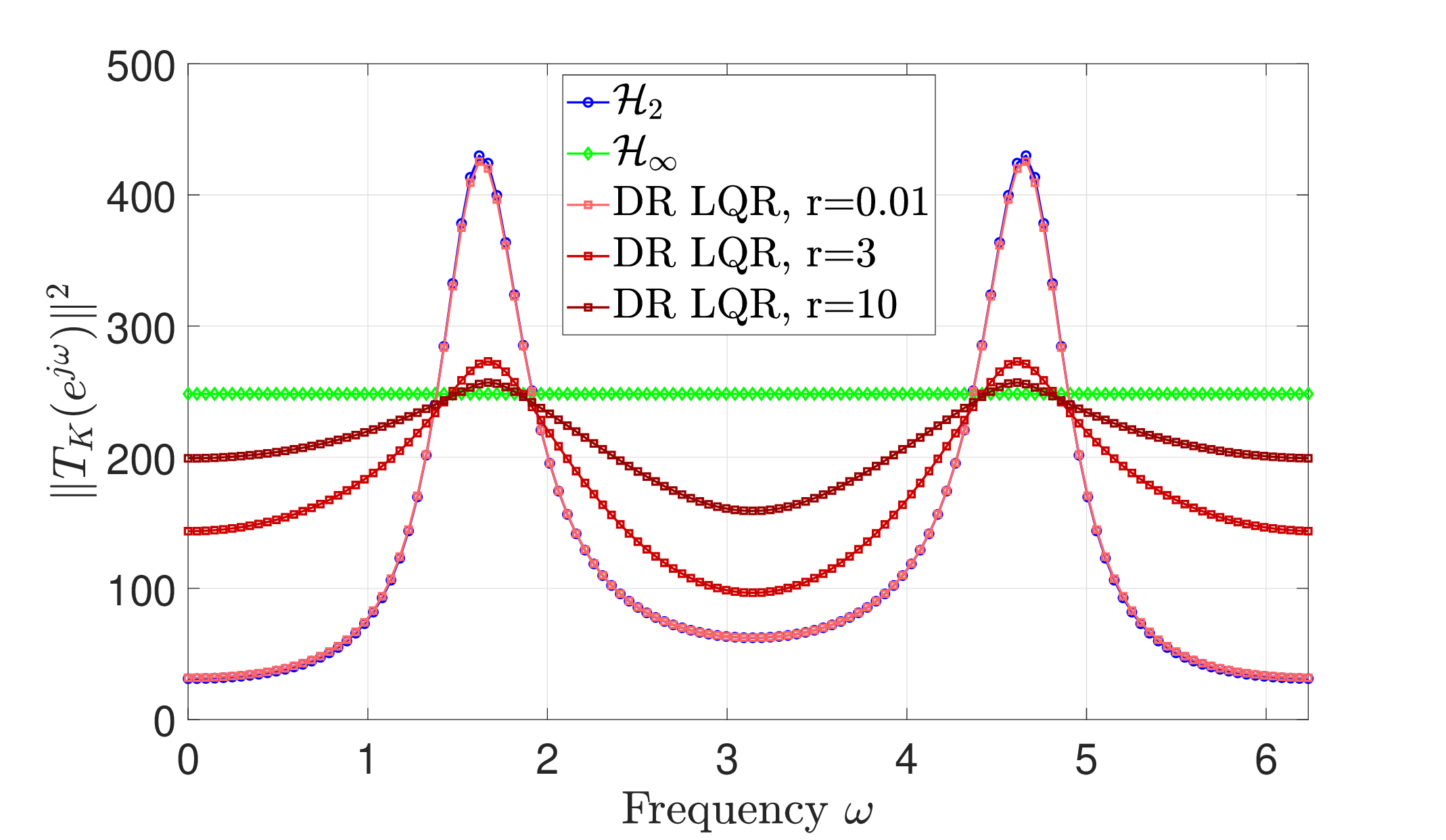}
    \caption{The operator norm, $\|T_K^\ast(\ejw)T_K(\ejw)\|$, of different controllers at all frequencies $\omega \in [0,2\pi]$, for system [AC15]. The DR-LQR cost interpolates between $H_2$ and $H_\infty$ based on the value of $r$. When $r$ is small (large), DR closely aligns with $H_2$ ($H_\infty$) \emph{across all frequencies}. 
    }
    \label{fig:R}
    \vspace{-5mm}
\end{figure}
To practically implement the DR-LQR controller, we find the rational controller by employing the method outlined in Section \ref{sec:rational_approx}, from which we obtain the rational approximation of $N(\ejw)$ as $\frac{P(\ejw)}{Q(\ejw)}$ with degrees $m=1,2$, for the [AC15] system. Table \ref{table:worstregret} compares the performance of these resulting rational controllers to the non-rational DR-LQR. Notably, the rational approximation with an order of 2 achieves an expected LQR cost that closely matches that of the non-rational controller with a difference of less than 1\% in costs for all $r$ values.

\begin{table}[htbp]
    \centering
    \setlength\tabcolsep{4pt} 
    \begin{tabular}{|c||c|c|c|c|c|c|c|} 
        \hline
        \textbf{ } & \textbf{r=0.01} & \textbf{r=1.5} & \textbf{r=5} & \textbf{r=10}  \\
        \hline \hline
        DR-LQR &\textbf{17.2} & \textbf{153.2} & \textbf{1024} & \textbf{3635.6}  \\
        \hline
        \textbf{RA(1)} & 17.2 &5789.1 & 6214.6 & 33262 \\
        \hline
        \textbf{RA(2)} & \textbf{17.19}   & \textbf{153.2}&\textbf{1024.1}& \textbf{3645.8} 
  \\
        \hline
    \end{tabular}
        \caption{The worst-case expected LQR cost of the non-rational DR-LQR controller, compared to the rational controllers RA(1), and RA(2), obtained from degree 1, and 2 rational approximations to $N(e^{j\omega})$.}
    \label{table:worstregret}
\end{table}

Finally, in figure \ref{fig:cvg}, we plot the convergence ratio defined as:
\begin{equation}\label{eq:ratio}
    \texttt{Convergence Ratio}\defeq\frac{W_2(\mathcal{M}_1^{i+1},\mathcal{M}_2^{i+1})}{W_2(\mathcal{M}_1^{i},\mathcal{M}_2^{i})}
\end{equation} where $W_2(\mathcal{M}_1^{i},\mathcal{M}_2^{i})$ represents the Wasserstein-2 distance between $\mathcal{M}_1$ and $\mathcal{M}_2$ at iteration $i$. The plot shows that the iterates $\{\mathcal{M}_1^{i}\}_{i\geq 0}$ converge and the rate of convergence of the \FixAlg is exponential.

\begin{figure}[htbp]
	\centering		 
 \includegraphics[width=0.42\textwidth]{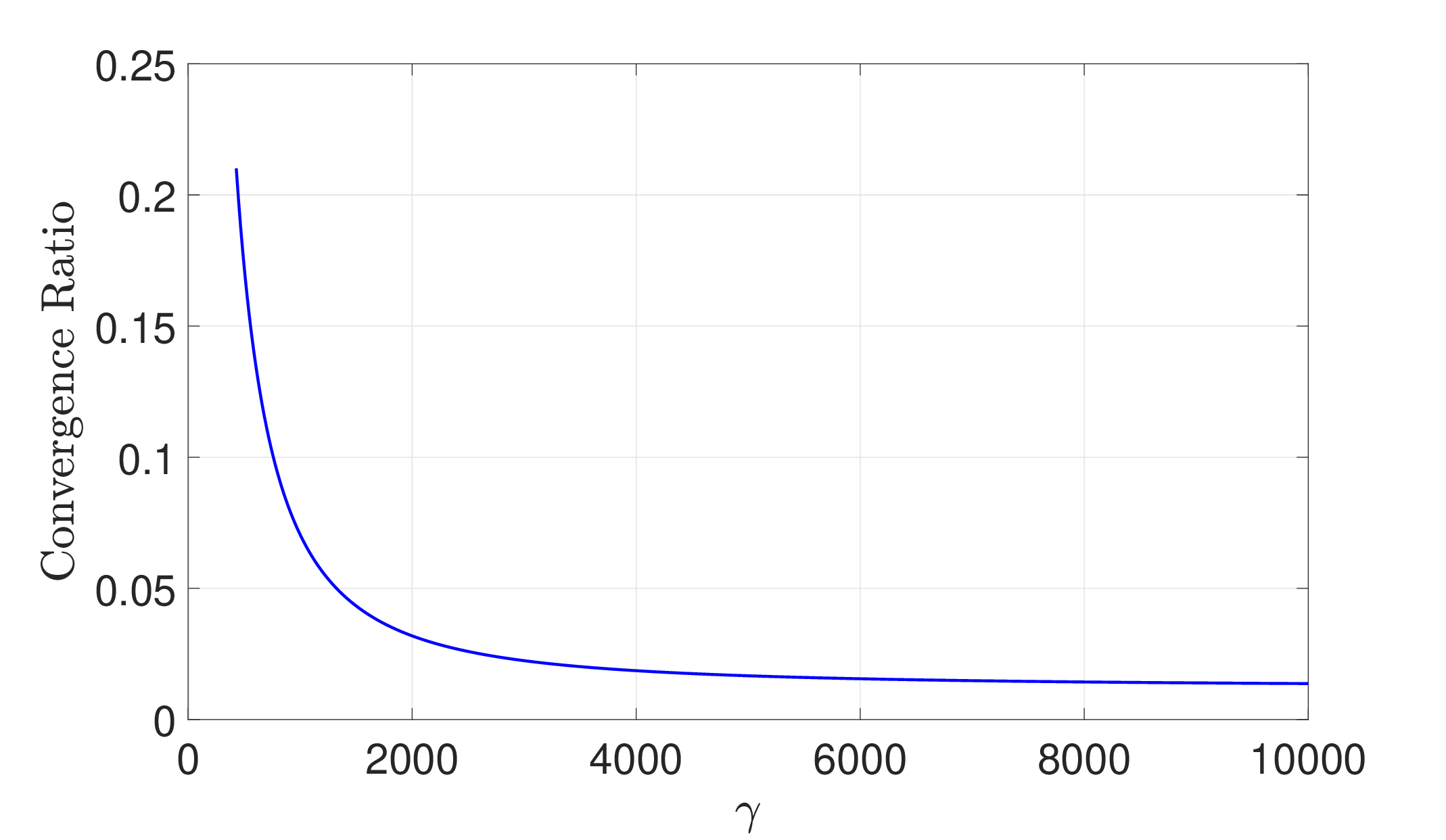}
       \caption{Convergence ratio \eqref{eq:ratio} for different values of $\gamma$. The \FixAlg algorithm converges at an exponential rate. 
        }
		\label{fig:cvg}
\end{figure}
\subsection{Time Domain Evaluations}

We compare the performance of the infinite horizon DR-LQR controller to its finite horizon counterpart, \emph{across time}. 

The finite-horizon DR-LQR, presented in \cite{cornell_drro_old}, is computed through an SDP whose dimension scales with the time horizon. We graph the mean LQR cost across 210 time steps, consolidating data from 1000 separate trials.
The performance of DR controllers under white Gaussian noise, and under worst-case disturbances of the infinite-horizon and the finite horizon DR controllers, is shown in figures \ref{fig:figa}, \ref{fig:figb} and \ref{fig:figc}, respectively.
For the sake of computational efficacy, the finite horizon controller operates within a constrained time horizon of $s=30$ steps, being recurrently applied every $s$ steps. Likewise, the worst-case disturbances used in figures \ref{fig:figb},\ref{fig:figc} are generated at the same periodicity, resulting in correlated disturbances solely within each $s$ step interval.

Our investigations underscore the unparalleled performance of the infinite horizon DR-LQR controller across all three scenarios. Note that attempting to extend the horizon of the SDP for prolonged durations to approximate the infinite horizon performance proves to be excessively computationally intensive. These findings accentuate the inherent advantages of adopting the infinite horizon DR-LQR controller.

\begin{figure}[htbp]
    \centering 
    \begin{subfigure}[b]{0.9\columnwidth}
        \centering
        \includegraphics[width=\columnwidth]{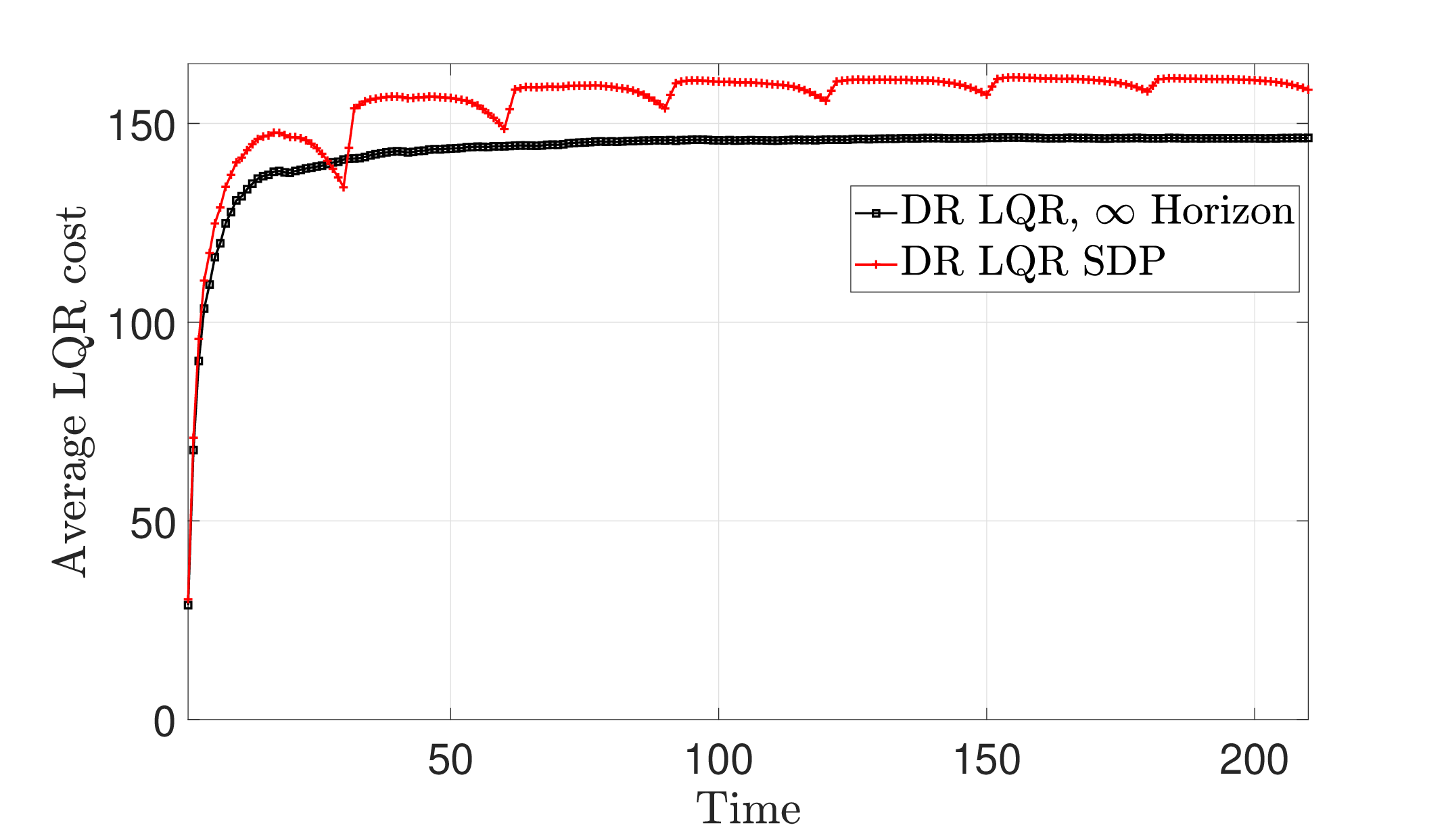}
        \caption{Control costs of DR-LQR controllers in the infinite (I) and finite horizon (II) under white noise.}
        \label{fig:figa}
    \end{subfigure}
    \vfill
    \begin{subfigure}[b]{0.85\columnwidth}
        \centering
        \includegraphics[width=\columnwidth]{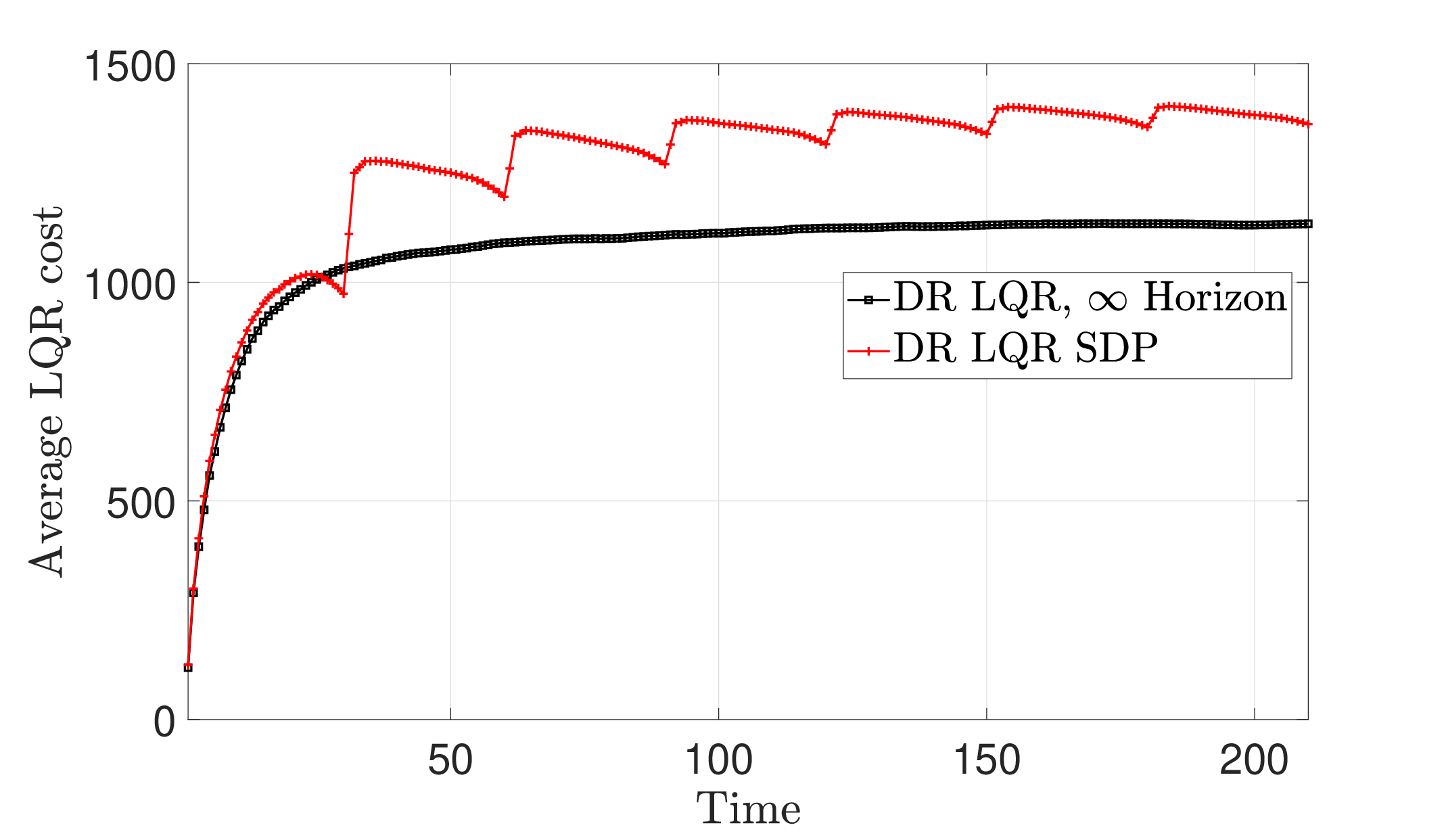}
        \caption{Control costs of DR-LQR controllers in the infinite (I) and finite horizon (II), under the worst-case disturbance for (I).}
        \label{fig:figb}
    \end{subfigure}
    \vfill
    \begin{subfigure}[b]{0.85\columnwidth}
        \centering
        \includegraphics[width=\columnwidth]{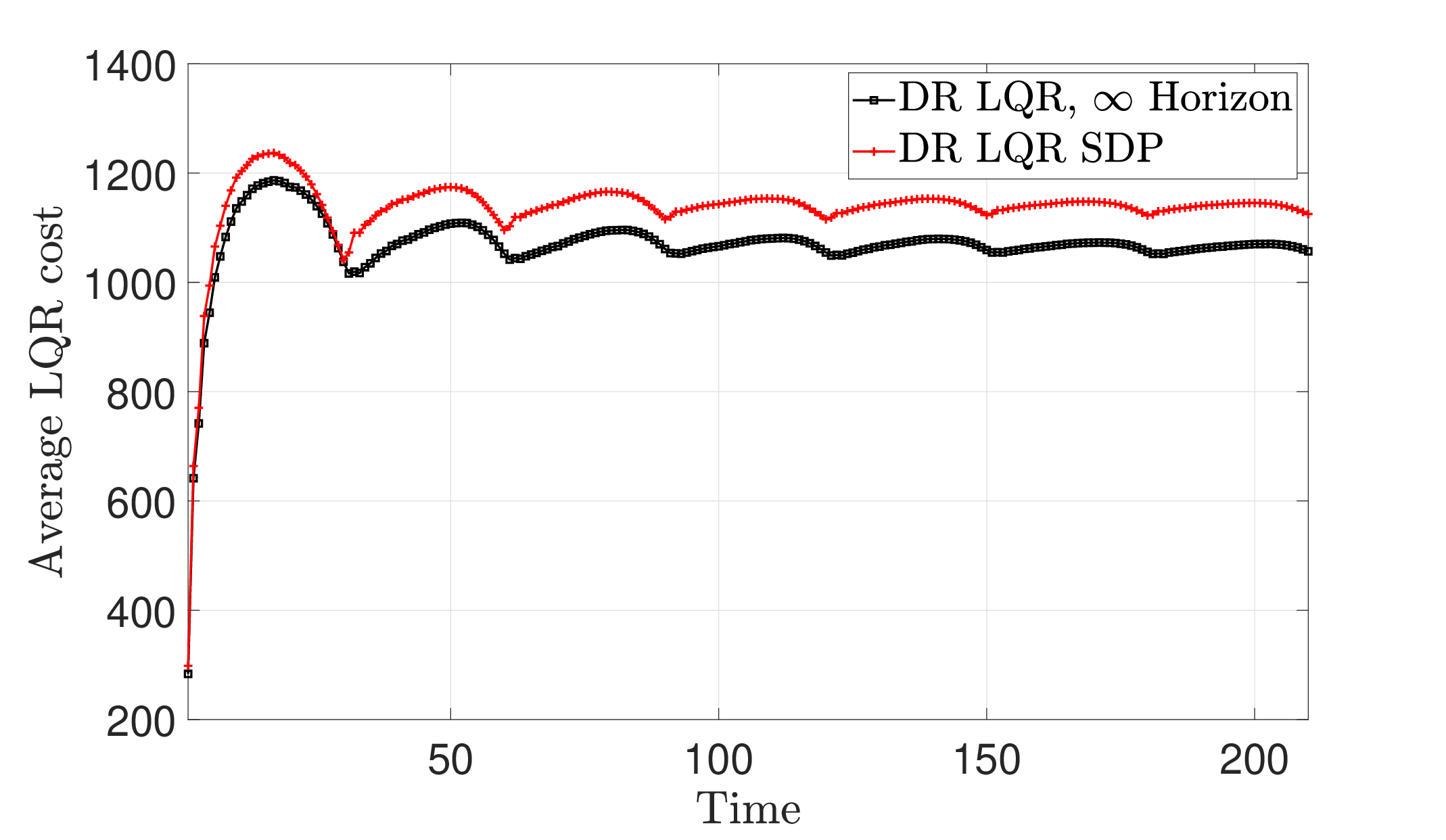}
        \caption{Control costs of DR-LQR controllers in the infinite (I) and finite horizon (II), under the worst-case disturbances for (II).}
        \label{fig:figc}
    \end{subfigure}
    \caption{Control costs of DR-LQR controllers in the infinite (I) and finite horizon (II) which is solved using SDP. The finite-horizon controller is recurrently applied at intervals of $s=30$ steps, and the radius of its ball is $r=1.5\sqrt{s}$, while the radius of (I) is $r=1.5$. The infinite horizon DR-LQR controller outperforms its finite-horizon counterpart and attains the minimum average cost in all cases, even when the finite horizon DR-LQR is designed to minimize the cost.}
\end{figure}

\section{Future Work}\label{sec:conclusion}


Our study proposes several overarching directions for future investigation. Firstly, we aim to prove the convergence of the fixed-point algorithm for the general case of non-scalar systems. Moreover, we plan to generalize the rational approximation method for cases where the disturbance is not scalar. Additionally, we seek to expand our findings and study partially observable systems (measurement-feedback case). 

\FloatBarrier

\bibliographystyle{IEEEtran}
\bibliography{references}

\begin{thebibliography}{10}
\providecommand{\url}[1]{#1}
\csname url@samestyle\endcsname
\providecommand{\newblock}{\relax}
\providecommand{\bibinfo}[2]{#2}
\providecommand{\BIBentrySTDinterwordspacing}{\spaceskip=0pt\relax}
\providecommand{\BIBentryALTinterwordstretchfactor}{4}
\providecommand{\BIBentryALTinterwordspacing}{\spaceskip=\fontdimen2\font plus
\BIBentryALTinterwordstretchfactor\fontdimen3\font minus \fontdimen4\font\relax}
\providecommand{\BIBforeignlanguage}[2]{{%
\expandafter\ifx\csname l@#1\endcsname\relax
\typeout{** WARNING: IEEEtran.bst: No hyphenation pattern has been}%
\typeout{** loaded for the language `#1'. Using the pattern for}%
\typeout{** the default language instead.}%
\else
\language=\csname l@#1\endcsname
\fi
#2}}
\providecommand{\BIBdecl}{\relax}
\BIBdecl

\bibitem{grinten1968uncertainty}
P.~M. E.~M. van~der Grinten, ``Uncertainty in measurement and control,'' \emph{Statistica Neerlandica}, vol.~22, no.~1, pp. 43--63, 1968.

\bibitem{doyle1985structured}
J.~C. Doyle, ``Structured uncertainty in control system design,'' in \emph{1985 24th IEEE Conference on Decision and Control}, 1985, pp. 260--265.

\bibitem{samuelson2017}
S.~Samuelson and I.~Yang, ``Data-driven distributionally robust control of energy storage to manage wind power fluctuations,'' in \emph{2017 IEEE Conference on Control Technology and Applications (CCTA)}, 2017, pp. 199--204.

\bibitem{kalman_new_1960}
R.~E. Kalman, ``\BIBforeignlanguage{en}{A {New} {Approach} to {Linear} {Filtering} and {Prediction} {Problems}},'' \emph{\BIBforeignlanguage{en}{Journal of Basic Engineering}}, vol.~82, no.~1, pp. 35--45, Mar. 1960.

\bibitem{zames_feedback_1981}
\BIBentryALTinterwordspacing
G.~Zames, ``\BIBforeignlanguage{en}{Feedback and optimal sensitivity: {Model} reference transformations, multiplicative seminorms, and approximate inverses},'' \emph{\BIBforeignlanguage{en}{IEEE Transactions on Automatic Control}}, vol.~26, no.~2, pp. 301--320, Apr. 1981. [Online]. Available: \url{http://ieeexplore.ieee.org/document/1102603/}
\BIBentrySTDinterwordspacing

\bibitem{doyle_state-space_1988}
\BIBentryALTinterwordspacing
J.~Doyle, K.~Glover, P.~Khargonekar, and B.~Francis, ``State-space solutions to standard $h_2$ and $h_\infty$ control problems,'' in \emph{1988 {American} {Control} {Conference}}.\hskip 1em plus 0.5em minus 0.4em\relax Atlanta, GA, USA: IEEE, Jun. 1988, pp. 1691--1696. [Online]. Available: \url{https://ieeexplore.ieee.org/document/4789992/}
\BIBentrySTDinterwordspacing

\bibitem{blackbook}
\BIBentryALTinterwordspacing
B.~Hassibi, A.~H. Sayed, and T.~Kailath, \emph{Indefinite-Quadratic Estimation and Control}.\hskip 1em plus 0.5em minus 0.4em\relax Society for Industrial and Applied Mathematics, 1999. [Online]. Available: \url{https://epubs.siam.org/doi/abs/10.1137/1.9781611970760}
\BIBentrySTDinterwordspacing

\bibitem{zhou_robust_1996}
K.~Zhou, J.~C. Doyle, and K.~Glover, \emph{Robust and optimal control}.\hskip 1em plus 0.5em minus 0.4em\relax Upper Saddle River, N.J: Prentice Hall, 1996.

\bibitem{yang2020wasserstein}
I.~Yang, ``Wasserstein distributionally robust stochastic control: A data-driven approach,'' \emph{IEEE Transactions on Automatic Control}, vol.~66, no.~8, pp. 3863--3870, 2020.

\bibitem{kim_distributional_2021}
\BIBentryALTinterwordspacing
K.~Kim and I.~Yang, ``Distributional robustness in minimax linear quadratic control with {Wasserstein} distance,'' Feb. 2021, arXiv:2102.12715 [cs, eess, math]. [Online]. Available: \url{http://arxiv.org/abs/2102.12715}
\BIBentrySTDinterwordspacing

\bibitem{hakobyan2022wasserstein}
A.~Hakobyan and I.~Yang, ``Wasserstein distributionally robust control of partially observable linear systems: Tractable approximation and performance guarantee,'' in \emph{2022 IEEE 61st Conference on Decision and Control (CDC)}.\hskip 1em plus 0.5em minus 0.4em\relax IEEE, 2022, pp. 4800--4807.

\bibitem{tacskesen2023distributionally}
B.~Taskesen, D.~A. Iancu, C.~Kocyigit, and D.~Kuhn, ``Distributionally robust linear quadratic control,'' \emph{arXiv preprint arXiv:2305.17037}, 2023.

\bibitem{aolaritei_wasserstein_2023}
\BIBentryALTinterwordspacing
L.~Aolaritei, M.~Fochesato, J.~Lygeros, and F.~Dörfler, ``Wasserstein {Tube} {MPC} with {Exact} {Uncertainty} {Propagation},'' Apr. 2023, arXiv:2304.12093 [math]. [Online]. Available: \url{http://arxiv.org/abs/2304.12093}
\BIBentrySTDinterwordspacing

\bibitem{aolaritei_capture_2023}
\BIBentryALTinterwordspacing
L.~Aolaritei, N.~Lanzetti, and F.~Dörfler, ``Capture, {Propagate}, and {Control} {Distributional} {Uncertainty},'' Apr. 2023, arXiv:2304.02235 [math]. [Online]. Available: \url{http://arxiv.org/abs/2304.02235}
\BIBentrySTDinterwordspacing

\bibitem{brouillon2023distributionally}
J.-S. Brouillon, A.~Martin, J.~Lygeros, F.~Dörfler, and G.~F. Trecate, ``Distributionally robust infinite-horizon control: from a pool of samples to the design of dependable controllers,'' 2023.

\bibitem{tzortzis_dynamic_2014}
\BIBentryALTinterwordspacing
I.~Tzortzis, C.~D. Charalambous, and T.~Charalambous, ``Dynamic {Programming} {Subject} to {Total} {Variation} {Distance} {Ambiguity},'' Feb. 2014, arXiv:1402.1009 [math]. [Online]. Available: \url{http://arxiv.org/abs/1402.1009}
\BIBentrySTDinterwordspacing

\bibitem{tzortzis_robust_2016}
\BIBentryALTinterwordspacing
I.~Tzortzis, C.~D. Charalambous, T.~Charalambous, C.~K. Kourtellaris, and C.~N. Hadjicostis, ``Robust {Linear} {Quadratic} {Regulator} for uncertain systems,'' in \emph{2016 {IEEE} 55th {Conference} on {Decision} and {Control} ({CDC})}.\hskip 1em plus 0.5em minus 0.4em\relax Las Vegas, NV, USA: IEEE, Dec. 2016, pp. 1515--1520. [Online]. Available: \url{http://ieeexplore.ieee.org/document/7798481/}
\BIBentrySTDinterwordspacing

\bibitem{liu_data-driven_2023}
\BIBentryALTinterwordspacing
R.~Liu, G.~Shi, and P.~Tokekar, ``Data-{Driven} {Distributionally} {Robust} {Optimal} {Control} with {State}-{Dependent} {Noise},'' Aug. 2023, arXiv:2303.02293 [cs]. [Online]. Available: \url{http://arxiv.org/abs/2303.02293}
\BIBentrySTDinterwordspacing

\bibitem{mohajerin_esfahani_data-driven_2018}
\BIBentryALTinterwordspacing
P.~Mohajerin~Esfahani and D.~Kuhn, ``\BIBforeignlanguage{en}{Data-driven distributionally robust optimization using the {Wasserstein} metric: performance guarantees and tractable reformulations},'' \emph{\BIBforeignlanguage{en}{Mathematical Programming}}, vol. 171, no. 1-2, pp. 115--166, Sep. 2018. [Online]. Available: \url{http://link.springer.com/10.1007/s10107-017-1172-1}
\BIBentrySTDinterwordspacing

\bibitem{gao_distributionally_2022}
\BIBentryALTinterwordspacing
R.~Gao and A.~J. Kleywegt, ``Distributionally {Robust} {Stochastic} {Optimization} with {Wasserstein} {Distance},'' Apr. 2022, arXiv:1604.02199 [math]. [Online]. Available: \url{http://arxiv.org/abs/1604.02199}
\BIBentrySTDinterwordspacing

\bibitem{DRORO}
\BIBentryALTinterwordspacing
F.~A. Taha, S.~Yan, and E.~Bitar, ``A {Distributionally} {Robust} {Approach} to {Regret} {Optimal} {Control} using the {Wasserstein} {Distance},'' in \emph{2023 62nd {IEEE} {Conference} on {Decision} and {Control} ({CDC})}.\hskip 1em plus 0.5em minus 0.4em\relax Singapore, Singapore: IEEE, Dec. 2023, pp. 2768--2775. [Online]. Available: \url{https://ieeexplore.ieee.org/document/10384311/}
\BIBentrySTDinterwordspacing

\bibitem{hajar_wasserstein_2023}
J.~Hajar, T.~Kargin, and B.~Hassibi, ``Wasserstein distributionally robust regret-optimal control under partial observability,'' in \emph{2023 59th Annual Allerton Conference on Communication, Control, and Computing (Allerton)}, 2023, pp. 1--6.

\bibitem{zhong2023nonlinear}
Z.~Zhong and J.-J. Zhu, ``Nonlinear wasserstein distributionally robust optimal control,'' \emph{arXiv preprint arXiv:2304.07415}, 2023.

\bibitem{kargin2023wasserstein}
\BIBentryALTinterwordspacing
T.~Kargin, J.~Hajar, V.~Malik, and B.~Hassibi, ``Wasserstein distributionally robust regret-optimal control in the infinite-horizon,'' 2023. [Online]. Available: \url{https://arxiv.org/abs/2312.17376}
\BIBentrySTDinterwordspacing

\bibitem{hajar_regret-optimal_2023}
\BIBentryALTinterwordspacing
J.~Hajar, O.~Sabag, and B.~Hassibi, ``Regret-optimal control under partial observability,'' 2023. [Online]. Available: \url{https://arxiv.org/abs/2311.06433}
\BIBentrySTDinterwordspacing

\bibitem{sabag2023regretoptimal}
\BIBentryALTinterwordspacing
O.~Sabag, G.~Goel, S.~Lale, and B.~Hassibi, ``Regret-optimal lqr control,'' 2023. [Online]. Available: \url{https://arxiv.org/abs/2105.01244}
\BIBentrySTDinterwordspacing

\bibitem{villani_optimal_2009}
C.~Villani, \emph{Optimal transport: old and new}, ser. Grundlehren der mathematischen {Wissenschaften}.\hskip 1em plus 0.5em minus 0.4em\relax Berlin: Springer, 2009, no. 338, oCLC: ocn244421231.

\bibitem{wassOT2}
F.~Santambrogio, ``Optimal transport for applied mathematicians,'' 2015.

\bibitem{sabag2021regret}
O.~Sabag, G.~Goel, S.~Lale, and B.~Hassibi, ``Regret-optimal controller for the full-information problem,'' in \emph{2021 American Control Conference (ACC)}.\hskip 1em plus 0.5em minus 0.4em\relax IEEE, 2021, pp. 4777--4782.

\bibitem{kailath_linear_2000}
T.~Kailath, A.~H. Sayed, and B.~Hassibi, \emph{Linear estimation}, ser. Prentice {Hall} information and system sciences series.\hskip 1em plus 0.5em minus 0.4em\relax Upper Saddle River, N.J: Prentice Hall, 2000.

\bibitem{rino_factorization_1970}
\BIBentryALTinterwordspacing
C.~Rino, ``\BIBforeignlanguage{en}{Factorization of spectra by discrete {Fourier} transforms ({Corresp}.)},'' \emph{\BIBforeignlanguage{en}{IEEE Transactions on Information Theory}}, vol.~16, no.~4, pp. 484--485, Jul. 1970. [Online]. Available: \url{http://ieeexplore.ieee.org/document/1054502/}
\BIBentrySTDinterwordspacing

\bibitem{kargin2024infinitehorizondistributionallyrobustregretoptimal}
\BIBentryALTinterwordspacing
T.~Kargin, J.~Hajar, V.~Malik, and B.~Hassibi, ``Infinite-horizon distributionally robust regret-optimal control,'' 2024. [Online]. Available: \url{https://arxiv.org/abs/2406.07248}
\BIBentrySTDinterwordspacing

\bibitem{sayed_survey_2001}
\BIBentryALTinterwordspacing
A.~H. Sayed and T.~Kailath, ``\BIBforeignlanguage{en}{A survey of spectral factorization methods},'' \emph{\BIBforeignlanguage{en}{Numerical Linear Algebra with Applications}}, vol.~8, no. 6-7, pp. 467--496, Sep. 2001. [Online]. Available: \url{https://onlinelibrary.wiley.com/doi/10.1002/nla.250}
\BIBentrySTDinterwordspacing

\bibitem{aircraft}
F.~Leibfritz and W.~Lipinski, ``Description of the benchmark examples in compleib 1.0,'' \emph{Dept. Math, Univ. Trier, Germany}, vol.~32, 2003.

\bibitem{cornell_drro_old}
S.~Yan, F.~A. Taha, and E.~Bitar, ``A distributionally robust approach to regret optimal control using the wasserstein distance,'' 2023.

\bibitem{ephremidze_algorithmization_2018}
\BIBentryALTinterwordspacing
L.~Ephremidze, F.~Saied, and I.~M. Spitkovsky, ``On the {Algorithmization} of {Janashia}-{Lagvilava} {Matrix} {Spectral} {Factorization} {Method},'' \emph{IEEE Transactions on Information Theory}, vol.~64, no.~2, pp. 728--737, Feb. 2018. [Online]. Available: \url{http://ieeexplore.ieee.org/document/8105834/}
\BIBentrySTDinterwordspacing

\end{thebibliography}

\vspace{-3mm}
\appendices
\onecolumn
\section{Definitions} \label{ap:org}
 \subsection{Parameters Definitions}\label{app1:def1} 
 We define $\overline{A},\overline{D}$ and $\overline{C}$ as:
 $\overline{A}\defeq A^\ast_K$, $\overline{D}\defeq A^\ast_K P B_w$, and $\overline{C}\defeq -{(R + B_u^\ast P B_u)}^{-\ast/2}B^\ast_u$ where: i)$A_K$ is the closed loop matrix $A_K\defeq A\-B_u K_{\textrm{lqr}}$, ii) $K_{\textrm{lqr}}$ is the LQR controller $K_{\textrm{lqr}}\!\defeq\!{(R\+B_u^\ast PB_u)}^{\inv}B^\ast_u P A$ and iii) $P \!\succ \!0$ is the unique stabilizing solution to the LQR Riccati equation $P=Q+A^\ast PA-A^\ast PB_u{(R+B_u^\ast PB_u)}^{-1}B_u^\ast PA$.

\subsection{Equations for the rational controller}\label{app:controller}
We give the equations for $\Tilde{F},\Tilde{G},\Tilde{H}$ and $\Tilde{J}$ of lemma \ref{lemma:rationalK}.

$\Tilde{F}=\begin{bmatrix}\Tilde{A_K} &0\\ B_u \Bar{R}^\ast \Bar{R} B_u^\ast & A_k\end{bmatrix}$,

$\Tilde{G}=\begin{bmatrix}\Tilde{A_K}\Tilde{B}\\-B_w+B_u\Bar{R}^\ast \Bar{R}B_u^\ast(P B_w+U \Tilde{B}) \end{bmatrix}$,

$\Tilde{H}=-R^{1/2}(\begin{bmatrix}
            \Bar{R}^\ast \Bar{R} B_u^\ast&-K_{lqr}
        \end{bmatrix})$,
        
$\Tilde{J}=-R^{1/2}(\Bar{R}^\ast \Bar{R} B_u^\ast (PB_w+U \Tilde{B})$.

Here, i) $K_{lqr}$, $A_K$ and $P$ are as defined in Appendix \ref{app1:def1}, ii) $\Bar{R}=(R+B_u^\ast P B_u)^{-\ast/2}$ 
iii) $\Tilde{A}_k=\Tilde{A}-\Tilde{B}\Tilde{C}$ where $\Tilde{A},\Tilde{B},\Tilde{C}$ as in lemma \ref{lemma:rationalK} and iv) $U$ satisfies the lyapunov equation $A_k^\ast P B_w \Tilde{C }+ A_k^\ast U A=U$.


\section{Proof of Optimality Theorem}\label{ap:pfopt}
\subsection*{Proof of Theorem \ref{thm:kkt}:} 
The $\gamma$-optimal LQR problem in \eqref{eq:suboptimal_prob} results from fixing $\gamma$ in \eqref{eq:full_optimality}.

The proof of the equations satisfied by the unique saddle point $(\mathcal{K}_{\gamma},\mathcal{L}_{\gamma})$ is as follows. 

First, we consider the following lemma:

    \begin{lemma}[{Wiener-Hopf Technique \cite{kailath_linear_2000}}]
    Consider the problem of approximating a non causal controller $\K_\circ$ by a causal controller $\K$, such that $\K$ minimises the cost $\Tr(\T_{\K}^\ast \T_{\K} \M )$, i.e.,
    \begin{align}
        \inf_{\K \in\causal} \Tr(\T_{\K}^\ast \T_{\K} \M )
    \end{align}
    where $\M \psdg 0$, $\T_{\K}^\ast \T_{\K}= \left(\K - \K_\circ \right)^{\ast}\Delta^{\ast}\Delta\left(\K - \K_\circ \right)+\T_{\K_\circ}^\ast \T_{\K_\circ}$, and $\K_\circ$ is the non-causal controller that makes the objective defined above zero. Then, the solution to this problem is given by 
    \begin{align}
        \K = \Delta^\inv\cl{\Delta \K_\circ \L}_{\!+} \L^\inv,
        \label{eq::nehari_controller}
    \end{align}
    where $\L$ is the unique causal and causally invertible spectral factor of $\M$ such that $\M = \L \L^\ast$ and $\cl{ \cdot }_{\!+}$ denotes the causal part of an operator.
    \label{lemma::nehari}
\end{lemma}
This lemma, whose proof is originally found in \cite{kailath_linear_2000}, serves as the proof of condition \eqref{eq:K from L in KKT} in our theorem \ref{thm:kkt}.

Second, we reformulate problem \eqref{eq:suboptimal_prob} as:
$$ \sup_{\M \psdg 0 }   \Tr(2\sqrt{\M}-\M) +\inf_{\K \in\causal}\gamma^\inv\Tr(\T_{\K}^\ast \T_{\K} \M ),$$

which is convex in $\mathcal{M}$. Hence, the optimal solution $\mathcal{M}_{\gamma}$ satisfies the following KKT conditions as in theorem 9 of \cite{kargin2023wasserstein}
\begin{align}
     \M_{\gamma}^{-\half} -\I + \gamma^\inv \T_{\K_\gamma}^\ast \T_{\K_\gamma} = 0.
    \label{eq::KKT_optimal_M}
\end{align}
Substituting the optimal controller $\K_{\gamma}$ with its equation \eqref{eq::nehari_controller}, and using $\T_{\K_{\gamma}}^\ast \T_{\K_{\gamma}}= \left(\K_{\gamma} - \K_\circ \right)^{\ast}\Delta^{\ast}\Delta\left(\K_{\gamma} - \K_\circ \right)+\T_{\K_\circ}^\ast \T_{\K_\circ}$ from Lemma \ref{lemma::nehari}, and using $\M_{\gamma}=\L_{\gamma}\L_{\gamma}^\ast$, we get,
\begin{align}
    -\I + (\L_{\gamma} \L_{\gamma}^\ast)^{-\half} + \gamma^\inv  \L_{\gamma}^{\!-\ast} \{\Delta\K_{\circ}\L_{\gamma}\}_{\!-}^\ast  \{\Delta\K_{\circ}\L_{\gamma}\}_{\!-} \L_{\gamma}^\inv + \gamma^\inv\T_{\K_\circ}^\ast \T_{\K_\circ}= 0.
\end{align}
Multiplying by $\L^\ast$ from the left and $\L$ from the right, we get,
\begin{align}
      &-\L_{\gamma}^\ast \L_{\gamma} + \L_{\gamma}^\ast(\L_{\gamma} \L_{\gamma}^\ast)^{-\half} \L_{\gamma} + \gamma^\inv \L_{\gamma}^\ast \L_{\gamma}^{\!-\ast} \{\Delta\K_{\circ}\L_{\gamma}\}_{\!-}^\ast  \{\Delta\K_{\circ}\L_{\gamma}\}_{\!-} \L_{\gamma}^\inv \L_{\gamma}+ \gamma^\inv(\T_{\K_\circ}\mathcal{L_\gamma})^\ast (\T_{\K_\circ}\mathcal{L_\gamma}) = 0.
\end{align}
which is equivalent to
\begin{align}
      -\L_{\gamma}^\ast \L_{\gamma} + (\L_{\gamma}^\ast \L_{\gamma} )^{\half} + \gamma^\inv \{\Delta\K_{\circ}\L_{\gamma}\}_{\!-}^\ast  \{\Delta\K_{\circ}\L_{\gamma}\}_{\!-}+\gamma^\inv(\T_{\K_\circ}\mathcal{L_\gamma})^\ast (\T_{\K_\circ}\mathcal{L_\gamma}) = 0.
\end{align}
Defining
\begin{align}
    &\mathcal S_{\L_\gamma} \defeq \{\Delta \K_\circ \L_\gamma\}_{\-} \quad \text{and}\quad \mathcal U_{\L_\gamma}\defeq \mathcal T_{\K_\circ}\L_\gamma,
\end{align}
we get through completing the squares:
\begin{align}
      \left( (\L_{\gamma}^\ast \L_{\gamma} )^{\half} - \frac{\I}{2} \right)^2 = \frac{\I}{4} + \gamma^\inv \left(\mathcal S_{\L_\gamma}^\ast \mathcal S_{\L_\gamma} + \mathcal U_{\L_{\gamma}}^\ast \mathcal U_{\L_{\gamma}}\right).
      \label{eq::kkt_square_form}
\end{align}

Taking the unique positive definite square root of the right hand side of \eqref{eq::kkt_square_form} concludes the proof.

\section{Proof of Convergence of Fixed Point in the Case of a Scalar System }\label{ap:pfcvg}

\subsection*{Proof of section \ref{sec:cvg}:} 
In this section, we provide a proof of convergence for the special case when $p = d=n = 1$. We first show  the fixed point that the Algorithm \ref{alg:fixed_point} converges to is unique when $\gamma \geq \ghi$. Then, we show that the iterates produce a sequence of monotonically increase spectrums $\NN$. Due to this, the algorithm must converge to the unique fixed point.

Consider the optimality condition \eqref{eq:N}. Note that for $\gamma > \ghi$, $N_\gamma(z)$ in \eqref{eq:N} is well defined. Thus, the map $F_{5,\gamma} \!\circ\! F_4 \!\circ\! F_3 \!\circ\! F_{2}\!\circ\!F_1 \!\circ\!F_6: N_\gamma(z) \mapsto N_\gamma(z)$ admits a fixed point $N_\gamma(z)$ for a fixed $\gamma$. Since \eqref{eq:suboptimal_prob} is concave in $\M$, the optimal solution $\NN^{*}$ is unique. Given that $M_\gamma(z)=L_\gamma(z)L_\gamma^\ast(z)$ (and $N_\gamma(z)= L_\gamma^\ast(z) L_\gamma(z)$), where $L_\gamma(z)$ represents a spectral factor of $M_\gamma(z)$ ($N_\gamma(z)$) that is both causal and causally invertible, it follows that $L_\gamma(z)$ is uniquely determined apart from a unitary transformation. By establishing a specific choice for the unitary transformation during the spectral factorization process, for example, opting for positive-definite factors at infinity as outlined by \cite{ephremidze_algorithmization_2018}, we ensure the uniqueness of $L_\gamma(z)$, and thus of $N_\gamma(z)$.

Let $\gamma > \ghi$. Consider now two spectrum $\NN_1 \preccurlyeq \NN_2$ which, in the frequency domain, are represented as $N_1(z) \leq N_2(z) \forall z$. Now, the spectrum $\NN_1, \NN_2$ are passed through one iteration of Algorithm \ref{alg:fixed_point}, \ie, $F_{2,\gamma} \!\circ\! (F_1\times \text{id}) \!\circ\! F_3 :\NN \mapsto \bar{\NN}$ to get $\bar{\NN_1}, \bar{\NN_2}$ respectively. We want to show that $\bar{n_1}(z) \leq \bar{n_2}(z) \forall z$ \ie, Algorithm \ref{alg:fixed_point} preserves the order of $\NN_1 \preccurlyeq \NN_2$. We have that,
\begin{align}
    \bar{N_1}(z) &=\frac{1}{4}\left(I \+ \sqrt{I \+4 \gamma^\inv \left(\{\Delta \K_{\circ}\L_{1}\}_{\!-}^\ast(z)  \{\Delta \K_{\circ}\L_{1}\}_{\!-}(z)+ L_{1}^\ast(z) T_{K_\circ}^\ast(z)  T_{K_\circ}(z) L_{1}(z)\right)}\right)^2 \label{eq:n1} \\    
    \bar{N_2}(z) &=\frac{1}{4}\left(I \+ \sqrt{I \+4 \gamma^\inv \left(\{\Delta \K_{\circ}\L_{2}\}_{\!-}^\ast(z)  \{\Delta \K_{\circ}\L_{2}\}_{\!-}(z)+ L_{2}^\ast(z) T_{K_\circ}^\ast(z)  T_{K_\circ}(z) L_{2}(z)\right)}\right)^2\label{eq:n2}
\end{align}

and,
\begin{align}
\Tr &\left(\{\Delta\K_{\circ}\L_{1}\}^{\ast}_{\!-}\{\Delta\K_{\circ}\L_{1}\}_{\!-} + \L_{1}^{\ast}\T_{\K_{\circ}}^{\ast}\T_{\K_{\circ}}\L_{1}\right) \nonumber \\
&= \inf_{\textrm{causal }\K } \Tr(\T_\K^\ast \T_\K \NN_1) \\
&\leq \inf_{\textrm{causal }\K } \Tr(\T_\K^\ast \T_\K \NN_2) \label{eq:n1leqn2} \\
&= \Tr \left(\{\Delta\K_{\circ}\L_{2}\}^{\ast}_{\!-}\{\Delta\K_{\circ}\L_{2}\}_{\!-} + \L_{2}^{\ast}\T_{\K_{\circ}}^{\ast}\T_{\K_{\circ}}\L_{2}\right),
\end{align}
where \eqref{eq:n1leqn2} is due to $\NN_1 \preccurlyeq \NN_2$. Moreover, since $\{\Delta\K_{\circ}\L_{2}\}^{\ast}_{\!-}\{\Delta\K_{\circ}\L_{2}\}_{\!-} + \L_{2}^{\ast}\T_{\K_{\circ}}^{\ast}\T_{\K_{\circ}}\L_{2} \succcurlyeq 0$, we have that,
\begin{align}
    \{\Delta\K_{\circ}\L_{1}\}^{\ast}_{\!-}&(z)\{\Delta\K_{\circ}\L_{1}\}_{\!-}(z) \leq \nonumber \{\Delta\K_{\circ}\L_{2}\}^{\ast}_{\!-}(z)\{\Delta\K_{\circ}\L_{2}\}_{\!-}(z) \forall z.
\end{align}
Hence, we have that, $\bar{N_1}(z) \leq \bar{N_2}(z)$ $\forall z$. We can now initialise the algorithm with $L(z) = 0, \forall z$. Since, after each iteration, $\bar{N}(z) \geq 0 \quad \forall z$, we have a monotonically increase sequence of $\bar{N}(z) \quad \forall z$ and the fixed point is unique, Algorithm \ref{alg:fixed_point} converges to the optimal $\bar{N}(z)$.

\textbf{Note} that even though we present a proof of convergence for a special case, empirical evidence suggests (as in section \ref{sec::numerical}) that the algorithm is exponentially convergent with a faster rate of convergence for larger $\gamma$.

\section{Proof of the State-Space Form of the DR-LQR controller }\label{ap:pfstate}

\subsection*{Proof of Lemma \ref{lemma:rationalK}:} 
Given the spectral factor $\Tilde L(z)$ in rational form as \begin{equation}\Tilde L(z)=(I+\Tilde{C} (z I -\Tilde{A})^{-1}\Tilde{B})\Tilde{D}^{1/2},\end{equation} its inverse given by:
\begin{align}\label{eq:linv}
    \Tilde L^{-1}(z)=\Tilde{D}^{-1/2}(I-\Tilde{C}(zI- (\Tilde{A}-\Tilde{B}\Tilde{C}))^{-1}\Tilde{B}),
\end{align}
and its operator form denoted by $\Tilde \L$.

The DR-RO controller, $K(z)$, is written as a sum of causal functions:
\begin{align}
      K(z)&=\Delta^{-1}(z) \{\Delta \K_\circ \Tilde \L\}_{+}(z)\Tilde L^{-1}(z)\\
      &=\Delta^{-1}(z) \left(\{\Delta \K_\circ \}_{+}(z)\Tilde L(z)+ \{\{\Delta \K_\circ \}_{-}\Tilde \L\}_{+}(z)\right)\Tilde L^{-1}(z)\\
      &=\Delta^{-1}(z) \{\Delta \K_\circ \}_{+}(z) + \Delta^{-1} \{\{\Delta \K_\circ \}_{-}\Tilde \L\}_{+}(z)\Tilde L^{-1}(z)\label{eq:last}.
\end{align}
From Lemma 4 in \cite{sabag2023regretoptimal}, we have:
\begin{align}\label{eq:dl}
    \{\Delta \K_\circ\}_{-}(z)=- \Bar{R} B_u^{\ast} (z^{\-1}I-A_k^\ast)^{-1}A_k^\ast P B_w
\end{align}

where all the terms have defined previously in Appendix \ref{app1:def1} and Appendix \ref{app:controller}.

Now, equation \eqref{eq:dl} is multiplied with $\Tilde L$, and the causal part of the result is:
\begin{align}
    \{\{\Delta \K_\circ\}_{-}\Tilde \L\}_{+}(z)= \{- \Bar{R} B_u^\ast (z^{\-1}I-A_k^\ast)^{-1}A_k^\ast P B_w \Tilde{C} (zI-\Tilde{A})^{-1}\Tilde{B}\Tilde{D}^{1/2}- \Bar{R}B_u^\ast (z^{\-1}I-A_k^\ast)^{-1}A_k^\ast P B_w \Tilde{D}^{1/2}  \}_{+}.
\end{align}

Let $\Tilde{U}$ be the matrix which solves the lyapunov equation: $A_k^\ast P B_w \Tilde{C }+ A_k^\ast \Tilde{U} A=\Tilde{U}$. Since $\Bar{R}B_u^\ast (z^{\-1}I-A_k^\ast)^{-1}A_k^\ast P B_w \Tilde{D}^{1/2}$ is strictly anticausal, $\{\{\Delta \K_\circ\}_{-}\Tilde \L\}_{+}(z)$ is as follows:

\begin{align}
    \{\{\Delta \K_\circ\}_{-}\Tilde \L\}_{+}(z)&=  \{- \Bar{R} B_u^\ast ((z^{\-1}I-A_k^\ast)^{-1}A_k^\ast \Tilde{U} + \Tilde{U} \Tilde{A}(zI-\Tilde{A})^{-1}+\Tilde{U})\Tilde{B}\Tilde{D}^{1/2} \}_{+}\\
    &=- \Bar{R} B_u^\ast \Tilde{U} (\Tilde{A}(zI-\Tilde{A})^{-1}+I)\Tilde{B}\Tilde{D}^{1/2}\\
    &=-z \Bar{R} B_u^\ast \Tilde{U} (zI-\Tilde{A})^{-1}\Tilde{B}\Tilde{D}^{1/2} \label{eq:dl_linv}
\end{align}

Now, equation \eqref{eq:dl_linv} is multiplied by the inverse of $\Tilde L$ \eqref{eq:linv}, the result is:
\begin{align}
    \{\{\Delta \K_\circ\}_{-}\Tilde \L\}_{+}(z)\Tilde L^{-1}(z)&=-z \Bar{R}B_u^\ast \Tilde{U} (zI-\Tilde{A})^{-1} \Tilde{B} (I+\Tilde{C}(zI-\Tilde{A})^{-1}\Tilde{B})^{-1}\\
    &=
    -z \Bar{R}B_u^\ast \Tilde{U} (zI-\Tilde{A})^{-1} (I+\Tilde{B}\Tilde{C}(zI-\Tilde{A})^{-1})^{-1}\Tilde{B}\\
    &=-z \Bar{R} B_u^\ast \Tilde{U} (zI-\Tilde{A}_k)^{-1}\Tilde{B}\\
    &=-\Bar{R }B_u^\ast \Tilde{U} (I+(zI-\Tilde{A}_k)^{-1}\Tilde{A}_k)\Tilde{B}
\end{align}
where $\Tilde{A}_k=\Tilde{A}-\Tilde{B}\Tilde{C}$.

The inverse of $\Delta$ is given by $\Delta^{-1}(z)=(I- K_{lqr} (zI-A_k)^{-1}B_u)\Bar{R}^\ast$, and lemma 4 in \cite{sabag2023regretoptimal} states that $\{\Delta \K_\circ\}_{+}(z)=-\Bar{R}B_u^\ast P A (zI-A)^{-1}B_w - \Bar{R} B_u^\ast P B_w$. 

Then the 2 terms of equation \eqref{eq:last} are: 

\begin{equation}\label{eq:1}
    \Delta^{-1}(z)\{\Delta \K_\circ\}_{+}(z)=-K_{lqr}(zI-A_k)^{-1}(B_w-B_u \Bar{R}^\ast \Bar{R}B_u^\ast P B_w)-\Bar{R}^\ast \Bar{R}B_u^\ast P B_w
\end{equation}

and 
\begin{align}\label{eq:2}
    \Delta^{-1}(z)\{\{\Delta \K_\circ\}_{-}\Tilde \L\}_{+}(z)\Tilde L^{-1}(z)&= - ( I-K_{lqr}(zI-A_k)^{-1}B_u )\Bar{R}^\ast \Bar{R}B_u^\ast \Tilde{U} (zI-\Tilde{A}_k)^{-1}\Tilde{A}_k \Tilde{B}\\
    &+ K_{lqr}(zI-{A}_k)^{-1}B_u \Bar{R}^\ast \Bar{R}B_u^\ast \Tilde{U} \Tilde{B}\\
    &-\Bar{R}^\ast \Bar{R}B_u^\ast \Tilde{U} \Tilde{B}
\end{align}

Finally, equations \eqref{eq:1} and \eqref{eq:2} are summed, and the resulting rational controller $K(z)$ is:
\begin{align}\label{eq:Kstatespace}
K(z)&=\underbrace{-\begin{bmatrix}
           \Bar{R}^\ast \Bar{R} B_u^\ast&-K_{lqr}
        \end{bmatrix}}_{\widetilde{H}} (zI-\underbrace{\begin{bmatrix}\Tilde{A_K} &0\\ B_u \Bar{R}^\ast \Bar{R} B_u^\ast & A_k\end{bmatrix}}_{\widetilde{F}})^{-1} \underbrace{\begin{bmatrix}\Tilde{A_K}\Tilde{B}\\-B_w+B_u\Bar{R}^\ast \Bar{R}B_u^\ast(P B_w+U_1 \Tilde{B}) \end{bmatrix}}_{\widetilde{G}} \nonumber \\
        &\underbrace{- \Bar{R}^\ast \Bar{R} B_u^\ast (PB_w+U_1 \Tilde{B})}_{\widetilde{J}}
\end{align}
which can be explicitly rewritten as in \eqref{eq:RationalK2}.

\newpage
\section{Implementation of Spectral Factorization via DFT}\label{ap:spectral factorization}

To perform the spectral factorization of an irrational function $N(z)$, we use a spectral factorization method via discrete Fourier transform, which returns samples of the spectral factor on the unit circle. The method is efficient without requiring rational spectra, and the associated error term, featuring a purely imaginary logarithm, rapidly diminishes with an increased number of samples. Note that this method is explicitly designed for scalar functions.

\begin{algorithm}[htbp]\label{alg:spectral factor method}
   \caption{\texttt{SpectralFactor}}
\begin{algorithmic}[1]
   \STATE {\bfseries Input:} Scalar positive spectrum $N(z)>0$ on  $\mathbb T_N \defeq \{\e^{j 2\pi n /N} \mid n \!=\! 0,\dots,N\!-\!1\}$ 
   \vspace{0.5mm}
   \STATE {\bfseries Output:} Causal spectral factor $L(z)$ of $N(z)>0$ on  $\mathbb T_N$ 
   \vspace{0.5mm}
   \STATE Compute the cepstrum $\begin{aligned}\Lambda(z) \gets \log(N(z))\end{aligned}$ on $z\in \mathbb T_N$.
   \vspace{0.5mm}
   \STATE Compute the inverse DFT  \\
    \vspace{0.5mm}
    $\begin{aligned}\lambda_k \gets \operatorname{IDFT}(\Lambda(z))\end{aligned}$ for $k=0,\dots,N\!-\!1$
   \vspace{0.5mm}
   \STATE Compute the spectral factor for $z_n= \e^{j 2\pi n /N}$ \\
    \vspace{0.5mm}
   $\begin{aligned}L(z_n) \gets \exp\pr{\frac{1}{2} \lambda_0 + \sum_{k=1}^{N/2-1} \lambda_k z_n^{-k} + \frac{1}{2} (-1)^{n} \lambda_{N/2}} \end{aligned}$, \quad $n=0,\dots,N\!-\!1$
\end{algorithmic}
\end{algorithm}

\end{document}